\documentclass[a4paper]{article}

\usepackage[english]{babel}
\usepackage{amsfonts}
\usepackage{enumitem}				
\usepackage{amsmath, amssymb}
\usepackage{mathtools}				
\usepackage{amsthm}					
\usepackage{bbm}					
\usepackage{graphicx}
\usepackage{xcolor}					
\usepackage{siunitx}				
\usepackage{soul}
\usepackage{subcaption}
\usepackage[round]{natbib}			
\usepackage[colorlinks=true, linkcolor=black, citecolor=black, hidelinks]{hyperref}

\title{\Large \bf Nonparametric two sample test of spectral densities}
\author{
	Ilaria Nadin\thanks{Department of Statistics and Operations Research, University of Vienna, Oskar-Morgenstern-Platz 1, 1090 Wien, Austria}
	\and
	Tatyana Krivobokova\thanks{Department of Statistics and Operations Research, University of Vienna, Oskar-Morgenstern-Platz 1, 1090 Wien, Austria}
	\and
	Farida Enikeeva\thanks{Laboratory of Mathematics and Its Applications, UMR CNRS 7348, University of Poitiers, 11 boulevard Marie et Pierre Curie, 86073 Poitiers Cedex 9, France}
}
\date{\today}

\graphicspath{{./art/}}

\newcommand{\norm}[1]{\left\lVert#1\right\rVert} 	
\newcommand{\lnorm}[1]{\lVert#1\rVert} 				
\newcommand{\abs}[1]{\left\lvert#1\right\rvert}		
\newcommand{\labs}[1]{\lvert#1\rvert}				
\newcommand{\floor}[1]{\lfloor#1\rfloor}			
\newcommand{\T}{\mathsf{T}}

\DeclareMathOperator{\tr}{\mbox{tr}}

\let\P\relax
\DeclareMathOperator{\Hyp}{\mathrm{H}}		
\DeclareMathOperator{\P}{\mathrm{P}}		
\DeclareMathOperator{\E}{\mathrm{E}}		
\DeclareMathOperator{\Var}{\mathrm{Var}}	
\DeclareMathOperator{\Cov}{\mathrm{Cov}}	

\theoremstyle{plain}
\newtheorem{theorem}{Theorem}
\newtheorem{lemma}{Lemma}

\newtheorem{proposition}{Proposition}
\newtheorem{corollary}{Corollary}

\theoremstyle{definition}
\newtheorem{definition}{Definition}
\newtheorem{assumption}{Assumption}

\newenvironment{keywords}{\vspace{1em} \noindent\textit{Keywords:}\ }{\par\vspace{1em}}
\newcommand{\step}[1]{\medskip\noindent\textit{Step~#1.} }

\begin{document}
	
	\maketitle
	
	\begin{abstract}
		
		A novel nonparametric test for the equality of the covariance matrices of two Gaussian stationary processes, possibly of different lengths, is proposed. The test translates to testing the equality of two spectral densities and is shown to be minimax rate-optimal. Test performance is validated in a simulation study, and the practical utility is demonstrated in the analysis of real electroencephalography data. The test is implemented in the R-package {\tt sdf.test}.
	\end{abstract}
	
	
	\section{Introduction}
	
	Testing whether two time series have the same stochastic structure has been a key area of interest since the 1980s. This problem arises in many applications such as meteorology,
	financial market analysis, biology, neuroscience, etc. This problem can be treated in the frequency domain by comparing two spectral densities of the observed time series or in the time domain by comparing the corresponding autocovariances. Our work focuses on a non-parametric approach to testing of equality of two spectral densities. 
	
	One of the earliest contributions to this field was by \citet{coates1986tests}, who studied whether two independent stationary time series share the same spectral density, testing against the scale-shift alternative by using the log-ratio of periodograms as the test statistic.  Around the same time, \citet{swanepoel1986} introduced a test that utilized parametric estimators of the spectral densities from two independent autoregressive processes, applying a bootstrap approach to determine its critical value. One of the first nonparametric tests is described in \citet{diggle1991nonparametric}, where a test for comparing two normalized cumulative periodograms is proposed using the Kolmogorov-Smirnov or Cramér-von Mises type distance between them. Since then, subsequent research in this field has increased, incorporating various statistical methods to handle different types of time series data. 	
	\citet{Gomes_Droucihe2002} considered the problem of testing proportionality of two spectral densities of autoregressive processes, utilizing the ratio of parametric estimators of spectral densities as the test statistic. In the time domain context of comparison of two stationary linear time series, \citet{guo1999} proposed a non-parametric test of comparison of two first-order autoregressive processes. More recently, \citet{jin_wang2016} developed a two-sample test based on empirical autocorrelations, while \citet{grant_quinn_2017} used a parametric likelihood-ratio test for comparing two autoregressions.  All these methods for two-sample testing of time series in the frequency domain rely on heuristic proofs and/or are constrained by the parametric autoregressive structure of the underlying process, limiting their applicability.
	
	The nonparametric approach to testing the equality of spectral densities has been considered by many authors. \citet{dette_paparoditis2009} develop a test based on an appropriate distance between the nonparametrically estimated individual spectral densities and an overall ``pooled'' spectral density derived from the entire set of time series under consideration. This test is specifically designed for assessing the equality of multiple spectral densities in multidimensional time series. The authors derive the limiting distribution of the test statistic under the null hypothesis and investigate the test's power and its limiting distribution for fixed alternatives. Another nonparametric method for comparing two spectral densities of time series of equal length was recently proposed by \citet {Zhang_Tu2022}. The authors use the Anderson-Darling-type test statistic calculated for periodogram ratios and obtain its limiting distribution under the null hypothesis. However, both methods fall short in addressing a significant practical issue: they cannot accommodate time series of different lengths. 
	
	One of the first attempts to address time series of different sizes can be found in \citet{caiado2012tests}, where the authors propose a test based on a periodogram-distance introduced in \citet{caiado2009comparison}. However, this approach was shown to have low power \citep[see][]{jentsch2012note}. 
	Subsequent contributions include the work of \citet{lu2013fan}, who developed a test based on the Fourier coefficients of the log-ratio of periodograms, and  \citet{preuss_hildebrandt2013}, who introduced a test statistic estimating the $\mathcal L_2$ distance between two spectral densities.  \citet{decowski2015wavelet}  proposed a test based on the wavelet decomposition of the ratio of periodograms, which was later improved by \citet{li2018tests}. These works share at least one of two important drawbacks: they either rely on time series periodograms, which are known to be unbiased but inconsistent \citep[see e.g. Theorem 2.14 in ][]{FanYao2005}, or limit theoretical developments to a study of the asymptotic distribution of the statistic under the null hypothesis and the consistency of the test. 
	
	We propose a novel method for testing the equality of two spectral densities, which is based on the $\mathcal L_2$ distance between two non-parametric estimators of spectral densities. Our test statistic employs the non-parametric estimator introduced in~\citet{klockmann2024efficient}, which is shown to be minimax optimal on a H\"older class of functions. This non-parametric spectral density estimator has two practical advantages. First, it shows strong small-sample performance, using twice as many data points as the periodogram by relying on the discrete cosine transform rather than the Fourier transform. Second, it allows the spectral density estimator to be evaluated on a chosen number of sample points. As a result, the spectral densities obtained from time series of unequal length can always be evaluated on the same set of points, facilitating comparison of their underlying stochastic structure.
	
	We analyze the performance of our testing procedure within the minimax framework and obtain the conditions on the minimal detectable norm of the difference between two spectral densities. The obtained testing rate is minimax optimal for H\"older classes of spectral density functions; the provided computer simulations corroborate these findings. Finally, we apply our testing procedure to real electroencephalography (EEG) data. 
	
	\section{Problem Statement}
	
	We observe two independent zero-mean random vectors $\mathbf{X}_1$ and $\mathbf{X}_2$ of dimensions $n_1$ and $n_2$ following Gaussian distribution with covariance matrices $\Sigma_1$ and $\Sigma_2$ respectively, such that
	\[
	\mathbf{X}_1 = (X_{1,1},\dots,X_{1,n_1})^\T \sim \mathcal{N}_{n_1}(0_{n_1},\Sigma_1),
	\qquad
	\mathbf{X}_2 = (X_{2,1},\dots,X_{2,n_2})^\T \sim \mathcal{N}_{n_2}(0_{n_2},\Sigma_2).
	\]
	Throughout this paper, the index $k$ takes values in $\{1,2\}$, indicating observations related to the time series $\mathbf{X}_1$ or $\mathbf{X}_2$ accordingly.
	The random vectors $\mathbf{X}_k$ are realizations of centered stationary Gaussian processes $\big\{X_k(t)\big\}_{t\in\mathbb Z}$ with auto-covariance functions $\gamma_k(h) = \Cov \{X_k(t),X_{k}(t+h)\} $, $h\in\mathbb{Z}$, and spectral densities $f_k$ defined as the Fourier transform of the auto-covariance functions
	\[
	f_k(x) = \gamma_k(0) + 2 \sum_{h=1}^\infty \gamma_k(h) \cos(hx), \quad x \in [-\pi,\pi],
	\]
	whenever $\sum_{h\in\mathbb{Z}} \abs{\gamma_k(h)} < \infty$.
	The inverse Fourier transform implies that 
	\[
	\gamma_k(h) = \int_{0}^{1} f_k(\pi x) \cos (h \pi x) dx,
	\]
	resulting in a one-to-one correspondence between $\gamma_k$ and $f_k$.
	Hence, the covariance matrices $\Sigma_{k} = \bigl\{\gamma_k(i-j)\bigr\}_{i,j=1}^{n_k}$ are completely characterised by their spectral density, whenever it exists. 
	
	Our goal is to test whether the covariances of $\mathbf{X}_1$ and $\mathbf{X}_2$ are generated by the same auto-covariance function.
	This problem is equivalent to testing the equality of their respective spectral densities, namely
	\[
	\Hyp_0:\ f_1 \equiv f_2 \quad \mbox{against}\quad \Hyp_1:\ \norm{f_1-f_2}_{2} \ge r_{n_1,n_2}>0,
	\]
	where $\norm{f}_2= \{ (2\pi)^{-1} \int_{-\pi}^{\pi} f(x)^2 dx\}^{1/2}$. 
	
	We suppose that the spectral densities are bounded, strictly positive, and belong to the H\"older class defined by 
	\[
	\mathcal{H}_\eta (M_0,M_1) = \biggl\{f:[-\pi,\pi] \to \mathbb{R}^+:\   \norm{f}_\infty \le M_0,\ \lnorm{f^{\floor{\eta}}(\cdot+h)- f^{\floor{\eta}}(\cdot)}_\infty \le M_1 \abs{h}^{\eta-\floor{\eta}}\biggr\},
	\]
	where $\eta>0$. 
	Let $L_c = \{f : \inf_x f(x) \geq c\}$ for some $c > 0$. The considered spectral densities are assumed to belong to the set $\mathcal{F}_\eta = \mathcal{H}_\eta (M_0, M_1) \cap L_c$. 
	Thus, we can reformulate the hypotheses as follows
	\begin{equation}\label{eq:Problem}
		\Hyp_0: \ (f_1,f_2) \in \mathcal B_{n_1,n_2}^0
		\qquad \mbox{against} \qquad
		\Hyp_1: \ (f_1,f_2) \in \mathcal B_{n_1,n_2}^1(\rho_{n_1,n_2}),
	\end{equation}
	where $\mathcal B_{n_1,n_2}^0 = \{ (f,f) : f \in \mathcal{F}_{\eta} \} $, and 
	\[
	\mathcal B_{n_1,n_2}^1 (\rho_{n_1,n_2}) = \Bigl\{ (f_1,f_2) \in \mathcal{F}_{\eta_1} \times \mathcal{F}_{\eta_2} : \norm{f_1-f_2}_2 \geq \rho_{n_1, n_2}\Bigr\}.
	\]
	
	For a test $\psi_{n_1,n_2}:(\mathbf{X}_1,\mathbf{X}_2)\to\{0,1\}$ define the type I and type II errors by
	\[
	\begin{split}
		\alpha(\psi_{n_1,n_2}) & =  \sup_{(f_1,f_2) \in \mathcal B_{n_1,n_2}^0 } \P_{f_1,f_2} \{\psi_{n_1,n_2} = 1\} , \\
		\beta(\psi_{n_1,n_2},\rho_{n_1,n_2}) & = \sup_{(f_1,f_2) \in \mathcal B_{n_1,n_2}^1 (\rho_{n_1,n_2}) } \P_{f_1,f_2} \{\psi_{n_1,n_2}=0\}.
	\end{split}
	\]
	It is well known that if the probability measures under the null and alternative hypotheses are too close, then the power of any testing procedure will be close to the type I error, making the hypotheses hardly distinguishable. The idea is  to measure the performance of a test using the so-called separation radius $\rho_{n_1,n_2}$ between two hypotheses \citep[for more details see][]{Ingster&Suslina:2003}. We are interested in the conditions on the minimal radius that allow the separation of two hypotheses, such that the difference between the two spectral densities is detectable at given significance levels. 
	
	More precisely, for given $\alpha\in (0,1)$, we say that a test $\psi_{n_1,n_2}^*$ is of level $\alpha$ if $\psi_{n_1,n_2}^* \in \Psi_\alpha$, where
	\[
	\Psi_\alpha = \{ \psi_{n_1,n_2}^*:(\mathbf{X}_1,\mathbf{X}_2)\to\{0,1\} : \alpha(\psi_{n_1,n_2}) \leq \alpha \}.
	\]
	We use the following definition of the minimax separation rate \citep[see][]{Baraud:2002}.
	\begin{definition}\label{def:nonasymp_sep_rate}
		Let $\alpha,\beta\in(0,1)$ be given. We say that $\rho_{n_1,n_2}^*$ is  an $(\alpha,\beta)$-minimax separation boundary if 
		\[
		\rho_{n_1,n_2}^* =\inf_{\psi\in\Psi_\alpha}\rho_{n_1,n_2}(\alpha,\psi),
		\]
		where 
		$\rho_{n_1,n_2}(\alpha,\psi)=\inf \Bigl\{\rho>0:\ \beta(\psi,\rho_{n_1,n_2})\le \beta\Bigr\}$. 	
	\end{definition}
	The goal of minimax testing is to find the $(\alpha,\beta)$-minimax separation rate  $\varphi_n$ and two constants $C_*$ and $C^*$ possibly depending on the significance levels $\alpha$ and $\beta$ such that for any $n \ge 1$, 
	\[
	C_* \varphi_{n_1,n_2} \le \rho^*_{n_1,n_2} \le C^* \varphi_{n_1,n_2} .
	\]
	In the present work, we provide an asymptotic upper bound on the minimax separation boundary with the optimal separation rate for H\"older classes of smoothness $\eta \ge 1/2$ \citep[see][Section 6.2  for minimax optimality results for testing on H\"older classes]{Ingster&Suslina:2003}. More precisely, we will show that 
	\[
	\rho^*_{n_1,n_2} \le C^* \varphi_{n_1,n_2} \{1+o(1)\}, \quad n_1, n_2 \to \infty,
	\]
	with $\varphi_{n_1,n_2} = (n_1^{-1} + n_2^{-1} )^{2 \eta^* /(4 \eta^* + 1)}$, where $\eta^* = \min \{ \eta, \eta_1, \eta_2 \}$ for $\eta, \eta_1, \eta_2, \ge 1/2$ and $C^*>0$ is a constant independent of $n_1$ and $n_2$.
	
	\section{Testing procedure}
	
	The idea of the test is to estimate the log-spectral densities of the two time series and to reject the null hypothesis whenever the norm of their difference is sufficiently large.
	For the estimation of the log-spectral densities, we rely on the method proposed by \citet{klockmann2024efficient}, which is based on a data transformation that reformulates the log-spectral density estimation problem as a mean estimation problem in approximate Gaussian regression.
	One advantage of this method is that it converts the initial time series into a binned dataset of a specified sample size. This allows us to transform the two original time series of different lengths into two binned series of equal length, from which log-spectral densities can be estimated and directly compared.
	
	
	\subsection{Data transformation}
	
	We outline the data transformation from \citet{klockmann2024efficient} and adapt it to our setting.
	\begin{enumerate}
		\item {\bf Discrete cosine transform:} For each vector $\mathbf{X}_k$, define new observations $\mathbf{W}_k$, such that $W_{k,j}=\{(D_k)_{j}^\T \mathbf{X}_k\}^2$ for $j=1,\dots,n_k$, where $D_k$ is the $n_k\times n_k$ matrix of discrete cosine transform I and $(D_k)_{j}$ is its $j$-th column.
		Resulting $W_{k,j}$ are gamma-disributred as
		\[
		W_{k,j}\sim \Gamma \Bigg\{ \frac 12,2f_k(\pi x_{k,j}) +\mathcal{O} \bigg( \frac1{n_k}+\frac{\log n_k}{n_k^{\eta_k}} \bigg) \Bigg\},
		\]
		with $x_{k,j}=(j-1)/(n_k-1)$. Moreover, $W_{k,j}$, $1\le j\le n_k$ are asymptotically independent. 
		
		\item {\bf Binning:} Choosing $\nu_k \in (1-\min\{1,\eta_k\}/3,1)$ such that $n_1^{\nu_1} = n_2^{\nu_2} = T$, the transformed data $\mathbf{W}_k$ are binned as follows
		\[
		Q_{k,t}=\sum_{j=(t-1)n_k/T+1}^{tn_k/T} W_{k,j},\quad t=1,\dots,T,\quad k=1,2,
		\]
		obtaining two vectors $\mathbf{Q}_1 = (Q_{1,1},\dots,Q_{1,T})^\T$ and $\mathbf{Q}_2 = (Q_{2,1},\dots,Q_{2,T})^\T$ of the same dimension. 
		
		\item {\bf Variance-stabilizing transform:} 
		Set 
		\[
		Y_{k,t} = \frac{\log(Q_{k,t}/m_k)}{2^{1/2}},\quad t=1,\dots, T,
		\]
		where $m_k = n_k/T = n_k^{1-\nu_k}$. Since the spectral density is symmetric around $0$ and periodic on $[-\pi, \pi]$, these observations can be mirrored into $\widetilde{T} = 2T-2$ observations $Y_{k,T}, \dots, Y_{k,2}, Y_{k,1}, \dots, Y_{k,T-1}$. Renumerating them, and scaling the design points to $[0,1)$, the new observations are approximately independent Gaussian with 
		\[
		Y_{k,t} \stackrel{\mbox{\tiny approx.}}{\sim } \mathcal{N} \bigg(H_k\bigl\{f_k(x_t)\bigr\},\frac1{m_k}\bigg),\quad x_t = \frac{t-1}{\widetilde{T}-1},
		\]
		for $t=1,\dots, \widetilde{T}$,	where $H_k(z) = 2^{-1/2} \{\phi(m_k/2) + \log(2z/m_k)\}$ and $\phi$ is the digamma function.
		
		\item {\bf Data translation:} The results are shifted as follows
		\[
		Y_{k,t}^* = Y_{k,t} - \frac{1}{2^{1/2}} \bigg\{ \phi \Big(\frac{m_k}{2} \Big) - \log \frac{m_k}{2} \bigg\}, \quad t=1,\dots, \widetilde{T}.
		\]
	\end{enumerate}
	
	The new observations $\mathbf{Y}_1^*=(Y_{1,1}^*,\dots,Y_{1,\widetilde{T}}^*)$ and $\mathbf{Y}_2^*=(Y_{2,1}^*,\dots,Y_{2,\widetilde{T}}^*)$ satisfy the following lemma, whose proof can be found in the Supplementary Material at the end.
	
	\begin{lemma}[Data decomposition]\label{lem:DataDecomposition}
		Consider $\mathbf{X}\sim\mathcal{N}_n(0_n,\Sigma)$, with spectral density $f\in\mathcal{F}_{\eta}$, $\eta > 0$.
		If $h > 0$ is such that $h \to 0$ and $hT \to \infty$, with $T=\floor{n^{\nu}}$ for $\nu \in (1-\min\{1,\eta\}/3,1)$, and $\mathbf{Y}^*$ is the vector of transformed data as in the proposed procedure, then $\mathbf{Y}^*$ can be decomposed in the following way
		\[
		\mathbf{Y}^* = \widetilde{\mathbf{Y}}^* + \overline{\mathbf{Y}},
		\]
		where 
		\[
		\widetilde{Y}_{t}^* = \frac{1}{2^{1/2}} \log f(x_t) + \epsilon_{t} + \zeta_{t} + \xi_{t}, \quad t=1,\dots,\widetilde{T},
		\]
		for $\abs{\epsilon_{t}} \leq c' (n^{-1} + n^{-\eta} \log n)$ small deterministic error with $c'>0$ constant, $\zeta_{t} \stackrel{\mathclap{iid}}{\sim} \mathcal{N}(0,m^{-1})$, $m=n/T$, standard Gaussian random variables, and $\xi_t$ centered independent random variables with decaying moments
		\[
		\E (\abs{\xi_t}^l) \leq c''_l (\log m)^{2l} \{ m^{-l} + (T^{-1} + T^{-1} n^{1-\eta} \log n)^l \}, \quad l > 1,
		\]
		for some constant $c''_l =c''_l(\eta, M_1)$ independent of $n$.
		Moreover,
		\[
		\norm{\overline{\mathbf{Y}}}_{\ell_2} = \mathcal{O}_p (m^{-1/2}),
		\]
		as $n \to + \infty$.
	\end{lemma}
	
	\subsection{Test statistics}
	
	Setting $g_k(x)= 2^{-1/2} \log \{f_k(x)\}$, we estimate it from the data $\mathbf{Y}_{k}^*$ with a periodic smoothing spline in an approximate Gaussian regression problem. The smoothing spline estimator $\widehat{g}_k$ is the solution of
	\[
	\min_{s \in \mathcal{S}_{\mathrm{per}}(2q-1,\underline{x}_{\widetilde{T}})} \Bigg[ \frac{1}{\widetilde{T}} \sum_{t=1}^{\widetilde{T}} \{ Y_{k,t}^* - s(x_t) \}^2 + h_k^{2q} \int_{0}^{1} \big\{ s^{(q)}(x) \big\}^2 dx \Bigg],
	\]
	where $h_k>0$ is the smoothing parameter, $q \in \mathbb{N}$ is the penalty order, and $\mathcal{S}_{\mathrm{per}}(2q-1,\underline{x}_{\widetilde{T}})$ is the space of periodic splines of degree $2q -1$ based on the equispaced knots $\underline{x}_{\widetilde{T}} = (x_t)_{i=0}^{\widetilde{T}} = (t/\widetilde{T})_{i=0}^{\widetilde{T}}$.
	The explicit form of the estimator can be written as 
	\begin{equation}\label{eq:density_est}
		\widehat g_k(x_t) 
		= \frac{1}{\widetilde{T} h_k}\sum_{s=1}^{\widetilde{T}} \mathcal{K}_{h_k,q} (x_t, x_s ) Y_{k,s}^*, \quad t=1, \dots, \widetilde{T},
	\end{equation}
	where the expression of the kernel $\mathcal{K}_{h_k,q}$ is given in the Supplementary Material \citep[see][for more information]{schwarz2016unified}.

	With this, we define the test statistic
	\[
	S(\mathbf{X}_1,\mathbf{X}_2) = \frac{1}{\widetilde{T}} \sum_{t=1}^{\widetilde{T}} \big\{ \widehat{g}_1(x_t) - \widehat{g}_2(x_t) \big\}^2 ,
	\]
	so that, for a given significance level $\alpha \in (0,1)$, the test becomes
	\begin{equation}\label{eq:DefPsi}
		\psi_{n_1,n_2}(\mathbf{X}_1, \mathbf{X}_2) = \mathbf{1} \{ S(\mathbf{X}_1,\mathbf{X}_2) > H_\alpha \},
	\end{equation}
	where $H_\alpha$ is the $(1-\alpha)$ quantile of $S$ under $\Hyp_0$. 
	
	\subsection{Main result}
	
	Let $\mathbf{X}_1 \sim \mathcal{N}_{n_1}(0_{n_1},\Sigma_1)$, $\mathbf{X}_2 \sim \mathcal{N}_{n_2}(0_{n_2},\Sigma_2)$, $n_1, n_2 \in \mathbb{N}$, be the observations of two zero-mean stationary Gaussian time series with spectral densities $f_1 \in \mathcal{F}_{\eta_1}$, $f_2 \in \mathcal{F}_{\eta_2}$, $\eta_1, \eta_2 \geq 1/2$, respectively.
	
	\begin{assumption}\label{assump:length}
		Assume that the lengths of the observed time series satisfy $1-\min\{1,\eta_2\}/3 \leq \log n_1 / \log n_2 \leq (1-\min\{1,\eta_1\}/3)^{-1}$. 
	\end{assumption}
	
	\begin{assumption}\label{assump:bandwidth}
		Let $T$ be the number of bins and the bandwidth $h_i>0$ of the spectral density estimatior of $f_i$ satisfy $h_i\to 0$, $h_iT\to\infty$ as $n_i\to\infty$, $i=1,2$. 
	\end{assumption}
	
	We need Assumption~\ref{assump:length} to be able to apply the data decomposition Lemma~\ref{lem:DataDecomposition}.	Indeed, if it is satisfied, we can set the number of bins to be $T = \floor{n_1^{\nu_1}} = \floor{n_2^{\nu_2}}$ with $\nu_k \in (1-\min\{1,\eta_k\}/3,1)$ for $k=1,2$. Assumption~\ref{assump:bandwidth} is a general non-parametric bandwidth assumption. 
	
	For the proposed test $\psi_{n_1,n_2}(\mathbf X_1,\mathbf X_2)$ defined in~\eqref{eq:DefPsi}, the following result holds.
	
	\begin{theorem}\label{th:MinimaxTest}
		
		Let $\alpha,\beta \in (0,1)$ be given significance levels. 	
		
		Let Assumptions~\ref{assump:length} and~\ref{assump:bandwidth} be satisfied and $q \in \mathbb{N}$ be the penalty order of the smoothing spline estimator~\eqref{eq:density_est}. 
		
		\begin{enumerate}[label=(\roman*)]
			\item Assume that under the null hypothesis $f_1=f_2=f \in \mathcal{F}_{\eta}$, where $\eta_1 = \eta_2 = \eta$. If 
			\begin{equation}\label{eq:OrderCritValue}
				H_\alpha  = \mathcal{C}_1 \sum_{k=1,2}\bigg\{ h_k^{2\min\{\eta, 2q\}} + \frac{h_k^{\min\{\eta, 2q\}}}{n_k^{1/2}} + \frac{1}{n_k h_k} +\frac{1}{n_k (\alpha h_k)^{1/2}} + \frac{1}{\alpha n_k} \bigg\} \{1 + o(1) \}
			\end{equation}
			for a constant $\mathcal{C}_1 = \mathcal{C}_1(\eta,q,c,M_0,M_1)$, then the test $\psi_{n_1,n_2}$  defined in~\eqref{eq:DefPsi} is of asymptotical level $\alpha$ as $n_1,n_2 \to \infty$.
			\item Set $\eta^* = \min \{\eta_1, \eta_2, \eta, q\}$. If 
			\begin{equation}\label{eq:ConditionNuk}
				\nu_k > - \frac{2 \eta^*}{ \floor{\min\{\eta_1, \eta_2\}}}\frac{\log h_k}{\log n_k},
			\end{equation}
			then  for a constant $\mathcal{C}_2 = \mathcal{C}_2 (\eta_1, \eta_2, \eta, q, \alpha, \beta, c, M_0, M_1)$ the $(\alpha,\beta)$-minimax separation boundary satisfies
			\begin{equation}\label{eq:radiusI}
				\rho_{n_1,n_2}^* \leq \mathcal{C}_2  \left\{h_1^{\eta^*} + h_2^{\eta^*} + \bigg( \frac{1}{\alpha^{1/2}} + \frac{1}{\beta^{1/2}} \bigg)^{1/2} \bigg( \frac{1}{n_1^{1/2} h_1^{1/4}} + \frac{1}{n_2^{1/2} h_2^{1/4}} \bigg)\right\} \{1+o(1)\}.
			\end{equation}
			\item Moreover, if $h_k = \mathcal{O} \big\{n_k^{-2/(4\eta^* +1)} \big\}$, then for the  critical value
			\[
			H_\alpha = \mathcal{C}_3 \Bigg\{ \bigg( \frac{1}{n_1} + \frac{1}{n_2} \bigg)^{\frac{4 \eta^* - 1}{ 4\eta^* +1}} \Bigg\}  \{ 1 + o(1) \},
			\]
			the $(\alpha,\beta)$-minimax separation boundary satisfies 
			\[
			\rho_{n_1,n_2}^* \leq \mathcal{C}_2\bigg( \frac{1}{\alpha^{1/2}} + \frac{1}{\beta^{1/2}} \bigg)^{1/2}\bigg( \frac{1}{n_1} + \frac{1}{n_2} \bigg)^{\frac{2 \eta^*}{ 4\eta^* +1}} \{1+o(1)\}
			\]
			for some constant $\mathcal{C}_3 = \mathcal{C}_3(\eta_1, \eta_2, \eta, q, \alpha,c,M_0,M_1)$. 
		\end{enumerate}
	\end{theorem}

	The proof of this theorem can be found in Appendix \ref{apx:ProofTheorem}.
	
	Note that if both time series are of the same length, that is, $n_1=n_2=n$, then some developments in Theorem \ref{th:MinimaxTest} simplify significantly, leading to the following corollary. 
	\begin{corollary}\label{cor:AdaptiveTest}
		If $n_1 = n_2=n$, then the results above hold for $\eta^* = \min \{\eta_1, \eta_2, \eta, 2q\}$. In particular, if $h_1 = h_2 = h = \mathcal{O} \big\{n^{-2/(4\eta^* +1)} \big\}$, then for the significance level 
		\[
		H_\alpha = \mathcal{C}_3 \, n^{ - \frac{4 \eta^* - 1}{ 4\eta^* +1}} \,  \{ 1 + o(1) \},
		\]
		the asymptotic upper bound on the $(\alpha,\beta)$-minimax separation boundary is given by 
		\[
		\rho_{n,n}^* \leq \mathcal{C}_2 \bigg( \frac{1}{\alpha^{1/2}} + \frac{1}{\beta^{1/2}} \bigg)^{1/2} \, n^{ - \frac{2 \eta^*}{ 4\eta^* +1}}(1+o(1))
		\]
		with the minimax separation rate $\varphi_{n,n}=n^{ - \frac{2 \eta^*}{ 4\eta^* +1}}$.
		
		More precisely, the test $\psi_{n,n}$ is of asymptotic level $\alpha$ and has the type II error asymptotically smaller than $\beta$ if under the alternative $\|f_1-f_2\|_2\gtrsim  \mathcal{C}_2 \big( \alpha^{-1/2} + \beta^{-1/2} \big)^{1/2} n^{ - \frac{2 \eta^*}{ 4\eta^* +1}}$.
	\end{corollary}

	%
	In particular, this implies that for $n_1=n_2$ our test is adaptive to the unknown smoothness parameters up to order $2q$. Precisely, if $\min \{\eta_1, \eta_2, \eta \} \leq 2q$ then the resulting separation radius attains the minimax rate associated with the underlying smoothness levels $\eta_1, \eta_2, \eta$, thereby adapting automatically to their unknown values. If instead $\min \{\eta_1, \eta_2, \eta \} \ge 2q$, then the separation radius is determined by $2q$, and is consequently governed by the tuning parameter $q$ rather than by the true smoothness. As a result, the procedure does not suffer from a loss of optimality due to undersmoothing: if $q$ is below the unknown smoothness level, as long as $2q$ exceeds the minimum smoothness, the minimax separation rate is still achieved.
	The proof of this Corollary \ref{cor:AdaptiveTest} is given in Appendix \ref{apx:ProofCorollary}.
	
	\section{Simulation study}
	\label{sec:SimulationStudy}
	
	In this section, we compare the power of our proposed test with the wavelet-based test proposed by \citet{decowski2015wavelet}.
	
	Further additional comparisons were conducted with the tests proposed by \citet{dette_paparoditis2009}, \citet{preuss_hildebrandt2013}, and \citet{caiado2012tests}. The tests of \citet{caiado2012tests} and \citet{dette_paparoditis2009} are based on suitably defined distance measures between two periodograms, with the latter being applicable only to the comparison of time series of equal lengths. Similarly, \citet{preuss_hildebrandt2013} proposed a test based on an $\mathcal{L}_2$-distance between exponentiated periodograms and derived the limiting distribution of the corresponding test statistic under the null hypothesis, thereby extending the later development of \citet{dette2011testing}, still restricted to time series of equal lengths, to the setting of unequal lengths.
	However, in our numerical simulations, these three tests consistently performed worse than our proposed test and the test by \citet{decowski2015wavelet} across the considered scenarios. As they did not yield additional insights, the corresponding results are deferred to the Supplementary Material, while we focus here only on the comparison with the wavelet-based test.
	
	We consider three main settings, for each of which we provide the power of the tests with significance level $\alpha = 0.05$ under two scenarios of unequal sample sizes: (A) one big and one small ($n_1=1200$ and $n_2=350$), and (B) both big ($n_1=1200$ and $n_2=1000$). In every setting, for $\delta=0,0.1,0.2, \dots,1$, we define the spectral density $f_2=f_{2,\delta}$ of the second process to be a convex combination of $f_1$ and a defined function $\widetilde{f}$, namely $f_{2,\delta}=(1-\delta)f_1+\delta \widetilde{f}$. The standard deviation is always assumed to be $\sigma = 1$, and therefore omitted.
	The settings are as follows:
	\begin{enumerate}
		\item In the first setting, we assume that $f_1$ is the spectral density of an autoregressive process AR(1) with parameter $\phi=0.5$, that is $f_1(x) = (2 \pi)^{-1} \{1 - 2 \phi \cos(x) + \phi^2\}^{-1}$ for $0 \leq x \leq \pi$. The alternative function $\widetilde{f}$ is the spectral density of an AR(1) process with parameter $\phi = 0.8$. Since $f_1$ is an analytic function, we use $q=6$. 
		\item In the second setting, we assume that $f_1(x) = (2 \pi)^{-1} (\abs{\cos( x/2 )}^{1.3} + 0.45)$ for $0 \leq x \leq \pi$, and the alternative function is the spectral density of an AR(2) process with parameters $\phi_1 = 0.3$, $\phi_2 = -0.5$, that is $\widetilde{f}(x) = (2 \pi)^{-1} \{1 + \phi_1^2 + \phi_2^2 + 2 \phi_1 (\phi_2 - 1) \cos( x) -2 \phi_2 \cos(2 x) \}^{-1}$ for $0 \leq x \leq \pi$. Since $f_1$ is Lipschitz, we use $q = 1$. 
		\item In the third setting we assume that $f_1(x)= (2 \pi)^{-1} \{\abs{\cos( x/2 )}^{5.1} + 0.45\}$, for $0 \leq x \leq \pi$, and the alternative function is $\widetilde{f}(x) = (2 \pi)^{-1} \{\abs{\cos( x/2 -0.2 \pi )}^{5.1} + 0.45\}$, for $0 \leq x \leq \pi$. Since $f_1\in\mathcal{F}_5$, we set $q=5$.
	\end{enumerate}
	For our test, the smoothing parameter for the log-spectral density is selected using generalised cross-validation (results with the smoothing parameter estimated by maximum likelihood are very similar and reported in the Supplementary materials).
	For both our test and the wavelet-based test by \citet{decowski2015wavelet}, we set the number of bins to $T=143$ for scenario (A), and  $T=143$ for scenario (B). For general discussion on the appropriate choice of all parameters, see \citet{klockmann2024efficient}.
	Additionally, for the wavelet-based test, we set $J=1$ because it yielded the best power, as suggested by \citet{decowski2015wavelet}.
	
	Since the distribution of our statistics under the null hypothesis depends on the unknown underlying spectral density, for each test, we compute the $(1-\alpha)$ quantile empirically via a Monte Carlo simulation as follows. Given two time series ${\mathbf{X}}_1$ and ${\mathbf{X}}_2$, we assume that both follow the same spectral density as the longest one, say ${\mathbf{X}}_1$. We estimate its spectral density using the data transformation procedure in \citet{klockmann2024efficient}, and use it to generate $M=5000$ Monte Carlo samples of pairs of time series of lengths $n_1$ and $n_2$ following a Gaussian distribution with such spectral density. The statistic is computed and collected for each pair, obtaining the empirical null distribution of the test. 
	
	For each $\delta$, we ran each test $500$ times. All the simulations are performed using \texttt{R} \citep[version 4.2.2, seed 42]{R} and the \texttt{R} package \texttt{sdf.test}, which implements our procedure and is available upon request.
	
	The obtained powers are shown in Figure \ref{fig:powers}. As expected, power curves in scenario (B) with longer time series are much steeper than those for scenario (A) for all settings under consideration. 
	Overall, both tests perform very similarly, but our test consistently (though slightly) outperforms the one by \citet{decowski2015wavelet}, albeit to a varying degree, across all cases. The type I errors of both tests are also very similar, with exact values given in the Supplementary materials. 
	\begin{figure}[!htb]
		\centering
		\begin{subfigure}{0.3\textwidth}
			\centering
			\includegraphics[width=\linewidth]{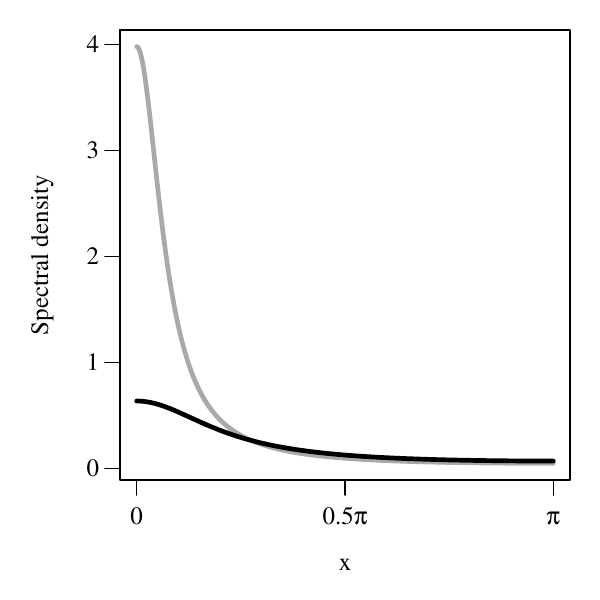}
			\caption{Setting 1, spectral densities}
			\label{fig:case1_sdf}
		\end{subfigure}
		\hfill
		\begin{subfigure}{0.3\textwidth}
			\centering
			\includegraphics[width=\linewidth]{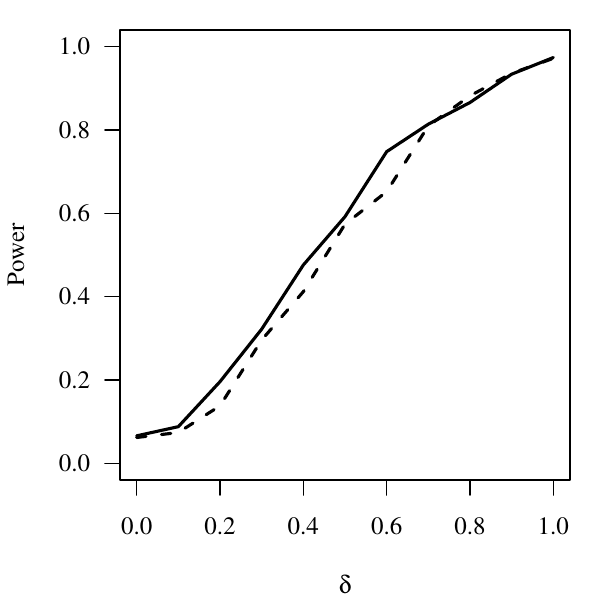}
			\caption{Setting 1, scenario (A)}
			\label{fig:case1_set1}
		\end{subfigure}
		\hfill
		\begin{subfigure}{0.3\textwidth}
			\centering
			\includegraphics[width=\linewidth]{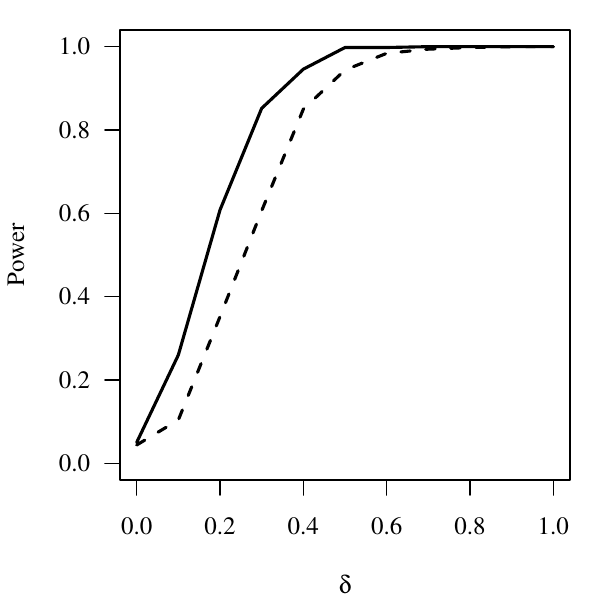}
			\caption{Setting 1, scenario (B)}
			\label{fig:case1_set2}
		\end{subfigure}
		
		\vspace{1em}
		
		\begin{subfigure}{0.3\textwidth}
			\centering
			\includegraphics[width=\linewidth]{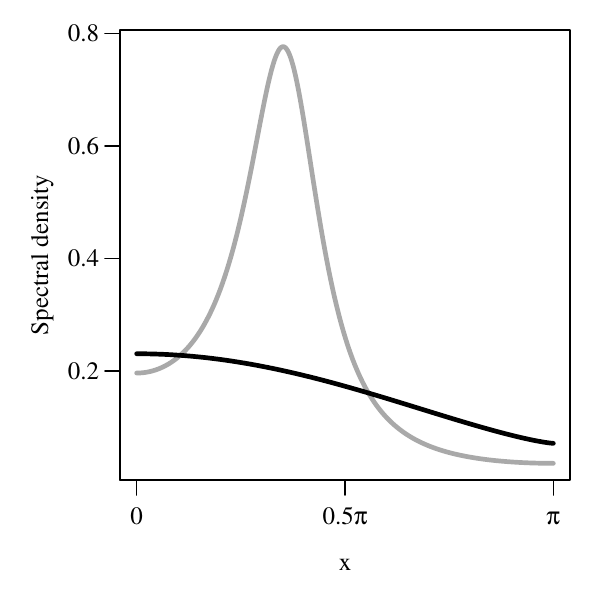}
			\caption{Setting 2, spectral densities}
			\label{fig:case2_sdf}
		\end{subfigure}
		\hfill
		\begin{subfigure}{0.3\textwidth}
			\centering
			\includegraphics[width=\linewidth]{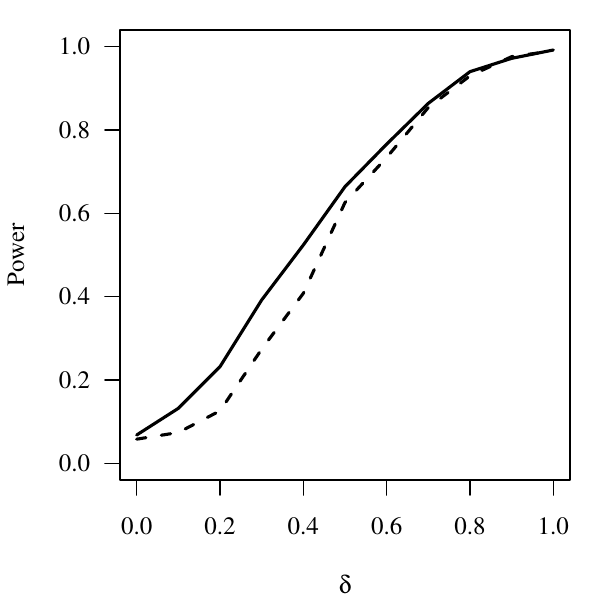}
			\caption{Setting 2, scenario (A)}
			\label{fig:case2_set1}
		\end{subfigure}
		\hfill
		\begin{subfigure}{0.3\textwidth}
			\centering
			\includegraphics[width=\linewidth]{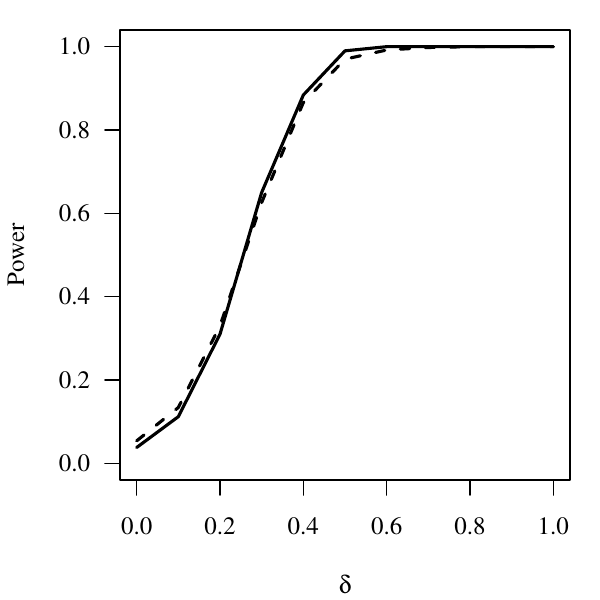}
			\caption{Setting 2, scenario (B)}
			\label{fig:case2_set2}
		\end{subfigure}
		
		\vspace{1em}
		
		\begin{subfigure}{0.3\textwidth}
			\centering
			\includegraphics[width=\linewidth]{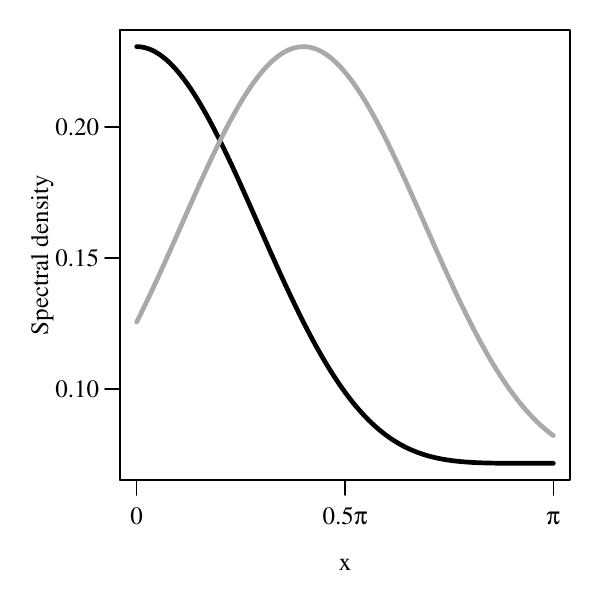}
			\caption{Setting 3, spectral densities}
			\label{fig:case3_sdf}
		\end{subfigure}
		\hfill
		\begin{subfigure}{0.3\textwidth}
			\centering
			\includegraphics[width=\linewidth]{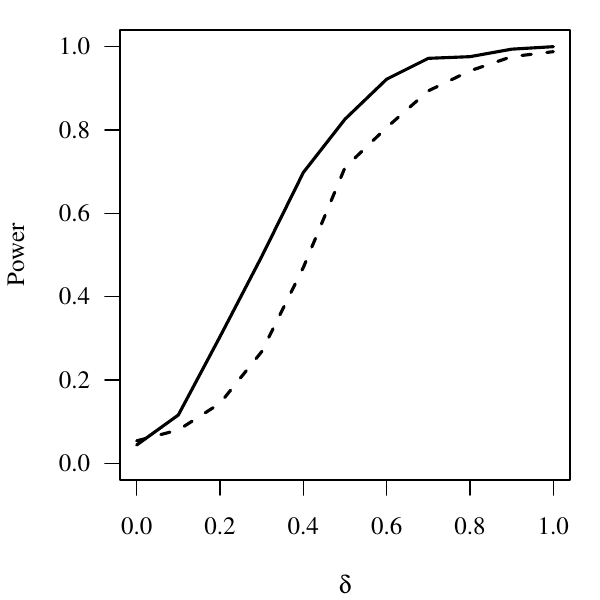}
			\caption{Setting 3, scenario (A)}
			\label{fig:case3_set1}
		\end{subfigure}
		\hfill
		\begin{subfigure}{0.3\textwidth}
			\centering
			\includegraphics[width=\linewidth]{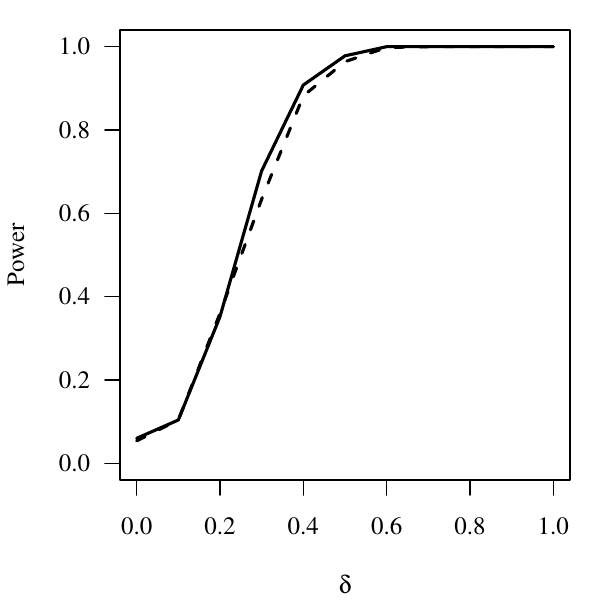}
			\caption{Setting 3, scenario (B)}
			\label{fig:case3_set2}
		\end{subfigure}
		
		\caption{In the first column, the spectral densities $f_1$ (black) and $\widetilde{f}$ (grey). In the second column, the powers of $\psi_{n_1,n_2}$ (solid), and the wavelet-based test (dashed) in scenario (A). In the third column, the powers of $\psi_{n_1,n_2}$ (solid), and the wavelet-based test (dashed) in scenario (B).}
		\label{fig:powers}
	\end{figure}
	
	
	
	
	\section{Application to EEG data}
	
	An electroencephalogram (EEG) is the recording of the brain's electrical activity, performed by several electrodes uniformly arrayed on the scalp. Each electrode captures the electric field generated by millions of neurons firing simultaneously. To measure the activity of a small area of the brain, one can take the difference between potentials measured at two electrodes, obtaining an EEG channel.
	
	We consider a dataset of preprocessed EEG signals collected by \citet{pastor2025deeplearningclassificationeeg}, available on Kaggle \citep[see][]{alexis_pomares_pastor_2021}. This database was generated by collecting the EEG signals of 11 healthy patients who underwent a monitored experiment involving alternating periods of rest and non-invasive brain stimulation. Neurostimulation was induced using two transcranial electrical stimulation (tES) techniques: direct current stimulation (tDCS) and alternating current stimulation (tACS). According to the method, an electrical current of 2 mA was directed bilaterally to one of two brain areas: the angular gyrus (in the posterior parietal lobes) or the middle frontal gyrus (in the frontal lobes). These two regions are often considered of interest for detecting conscious neural activity. In the original study, the purpose of the experiment was to develop a Deep Learning classification tool to determine whether the brain is responsive to the applied tES.
	The EEG signals were recorded using $188$ electrodes and subsequently preprocessed to remove common EEG artefacts, including non-neural physiological activity, direct-current (DC) baseline drift, and other electrical noise \citep[the full preprocessing pipeline is available on GitHub; see][]{pastor2025deeplearningclassificationeeg}.
	This preprocessing step is essential to ensure that the recorded signals are primarily due to neural activity, as EEG amplitudes are typically small and highly susceptible to contamination.
	As a result of this procedure, EEG segments corresponding to different stimulation periods may differ in length, even though the stimulation duration was the same. Therefore, comparing such signal segments is usually nontrivial, and our test should be a valid tool to overcome this problem.
	Our goal is to detect differences in the spectral density of EEG channels that may indicate changes in brain activity during frontal or posterior tDCS.
	To do so, we restrict our analysis to the recordings of one patient (P001). 
	
	We select a channel in the posterior area of the electrode cap, obtained by the potential difference between electrodes E159 and E160. We consider the time series of EEG signals from this channel in the following states:
	\begin{enumerate}
		\item Initial rest ($54$ \si{\second}), recorded in \texttt{run0\_20210726\_033258.csv}.
		\item First frontal tDCS ($30$ \si{\second}), recorded in \texttt{run3\_20210726\_035641.csv}.
		\item Posterior tDCS ($28$ \si{\second}), recorded in \texttt{run2\_20210726\_034709.csv}.
		\item Second frontal tDCS ($28$ \si{\second}), recorded in \texttt{run2\_20210726\_034709.csv}.
	\end{enumerate}	
	Each EEG signal had at most 2 observations that were more than 10 standard deviations from the data and proved very influential. Therefore, we removed those observations from the analysis. 
	
	We first checked the stationarity of each time series using the augmented Dickey-Fuller test. All series were found to be stationary linear, with p-values below $0.01$. The Hurst exponents of the time series fall approximately in the interval $(0.66, 0.72)$, suggesting moderate long-range dependence. The obtained signals have different lengths, a sampling rate of $250$ \si{\hertz}, and are shown in the first column of Figure \ref{fig:eeg-signals}.
	
	\begin{figure}
		\centering
		\begin{subfigure}{0.45\textwidth}
			\centering
			\includegraphics[width=\linewidth]{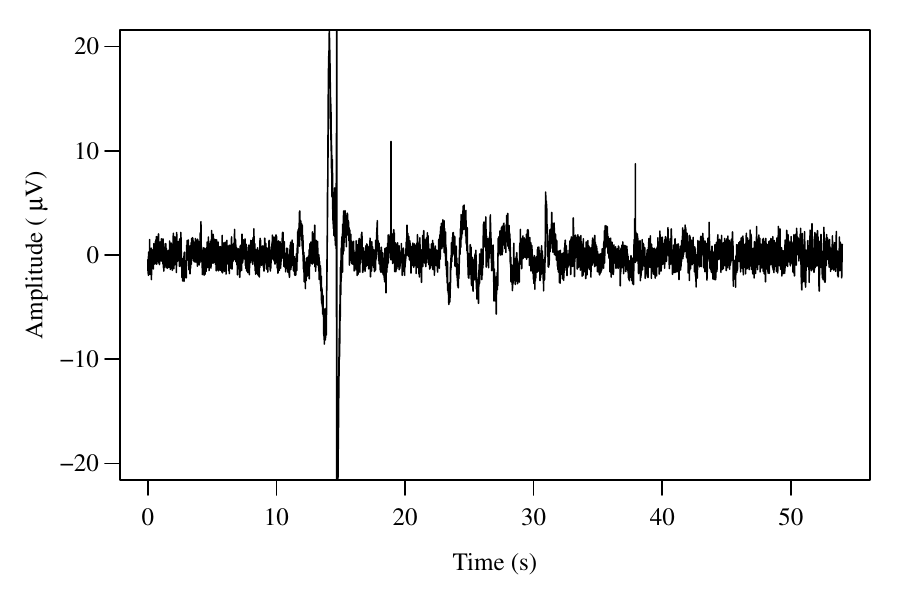}
			\caption{Initial rest}
			\label{fig:eeg1}
		\end{subfigure}
		\hfill
		\begin{subfigure}{0.45\textwidth}
			\centering
			\includegraphics[width=\linewidth]{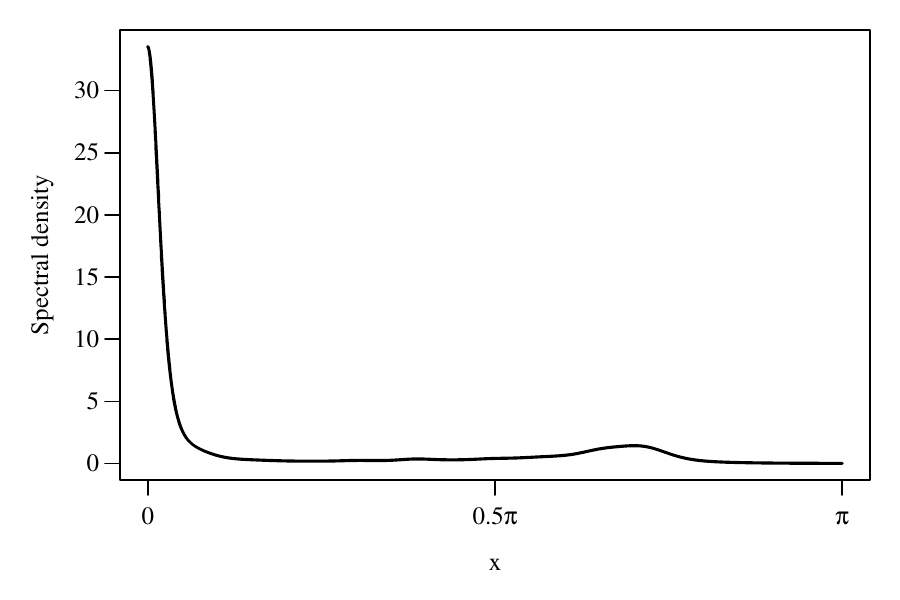}
			\caption{Initial rest spectral density}
			\label{fig:sdf1}
		\end{subfigure}
		
		\begin{subfigure}{0.45\textwidth}
			\centering
			\includegraphics[width=\linewidth]{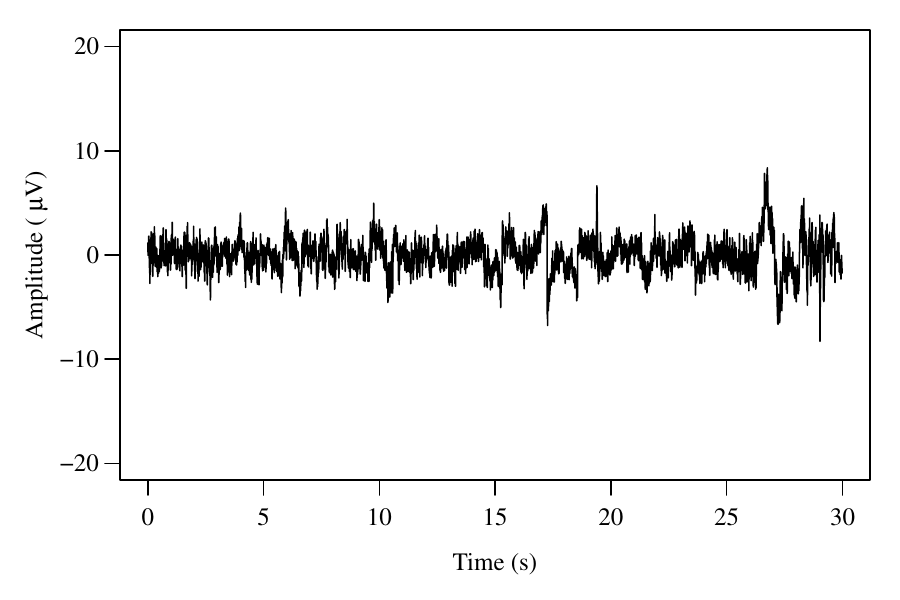}
			\caption{Frontal tDCS}
			\label{fig:eeg2}
		\end{subfigure}
		\hfill
		\begin{subfigure}{0.45\textwidth}
			\centering
			\includegraphics[width=\linewidth]{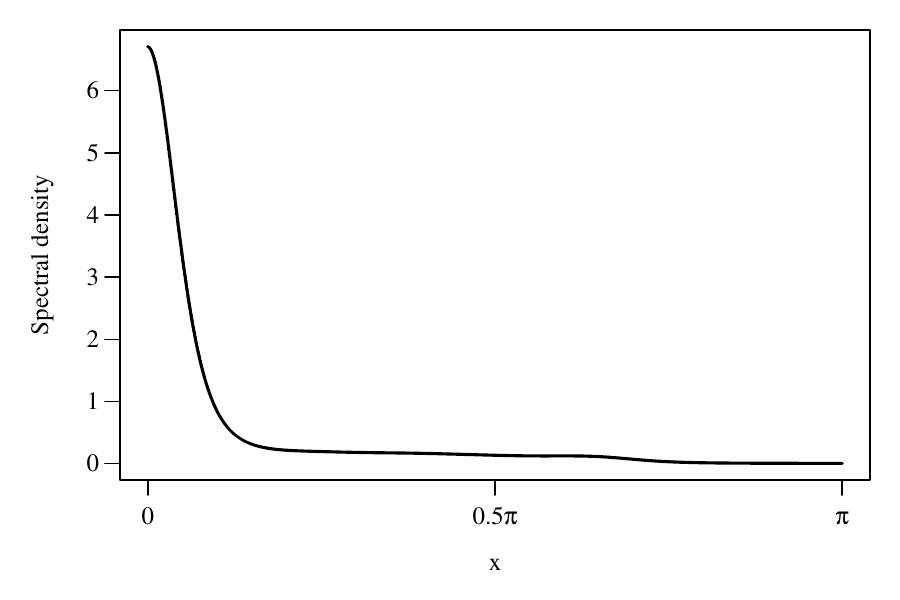}
			\caption{Frontal tDCS spectral density}
			\label{fig:sdf2}
		\end{subfigure}
		
		\begin{subfigure}{0.45\textwidth}
			\centering
			\includegraphics[width=\linewidth]{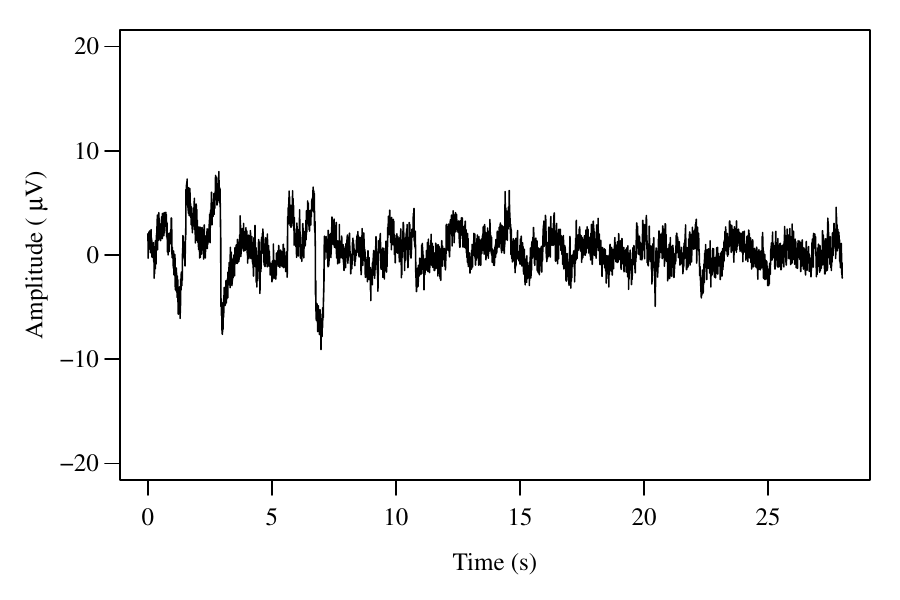}
			\caption{Posterior tDCS}
			\label{fig:eeg3}
		\end{subfigure}
		\hfill
		\begin{subfigure}{0.45\textwidth}
			\centering
			\includegraphics[width=\linewidth]{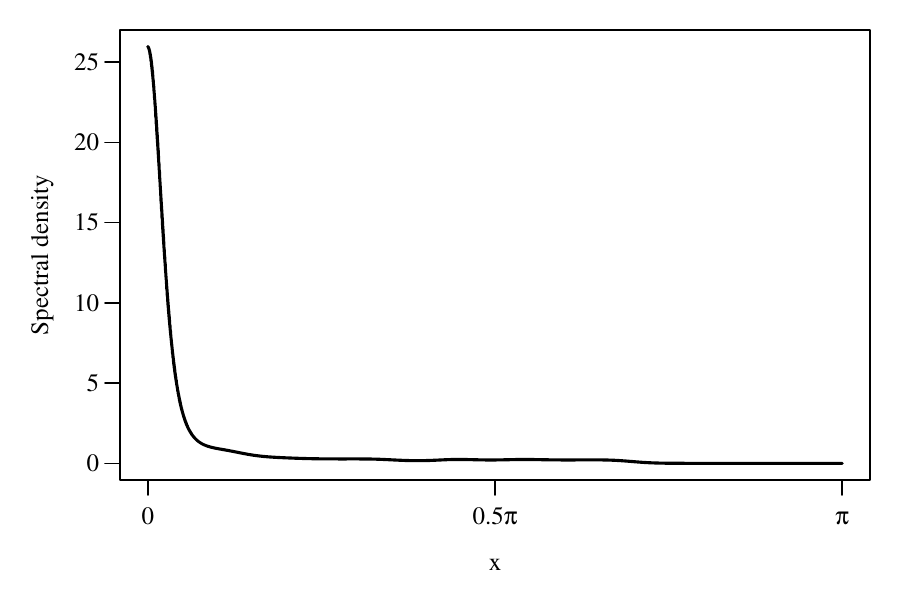}
			\caption{Posterior tDCS spectral density}
			\label{fig:sdf3}
		\end{subfigure}
		
		\begin{subfigure}{0.45\textwidth}
			\centering
			\includegraphics[width=\linewidth]{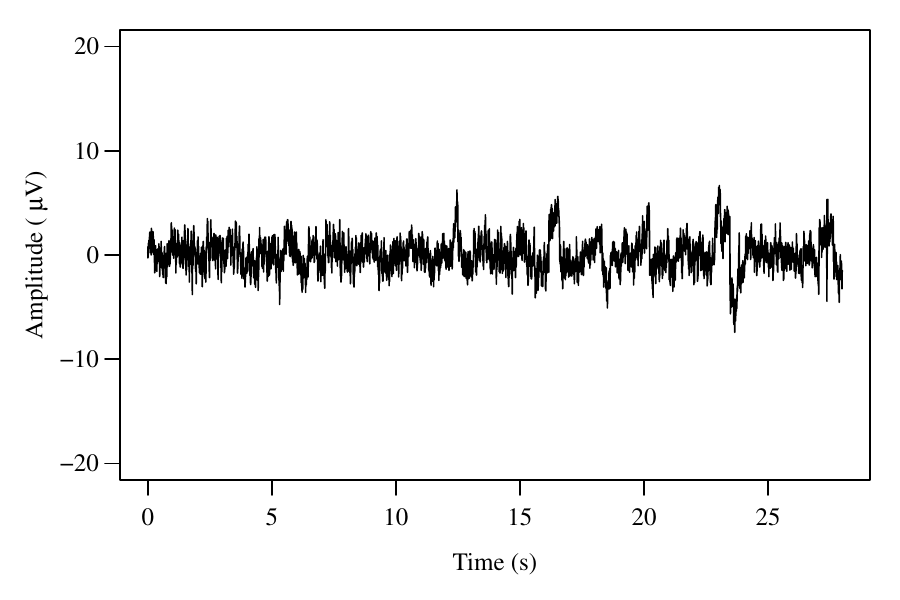}
			\caption{Second frontal tDCS}
			\label{fig:eeg4}
		\end{subfigure}
		\hfill
		\begin{subfigure}{0.45\textwidth}
			\centering
			\includegraphics[width=\linewidth]{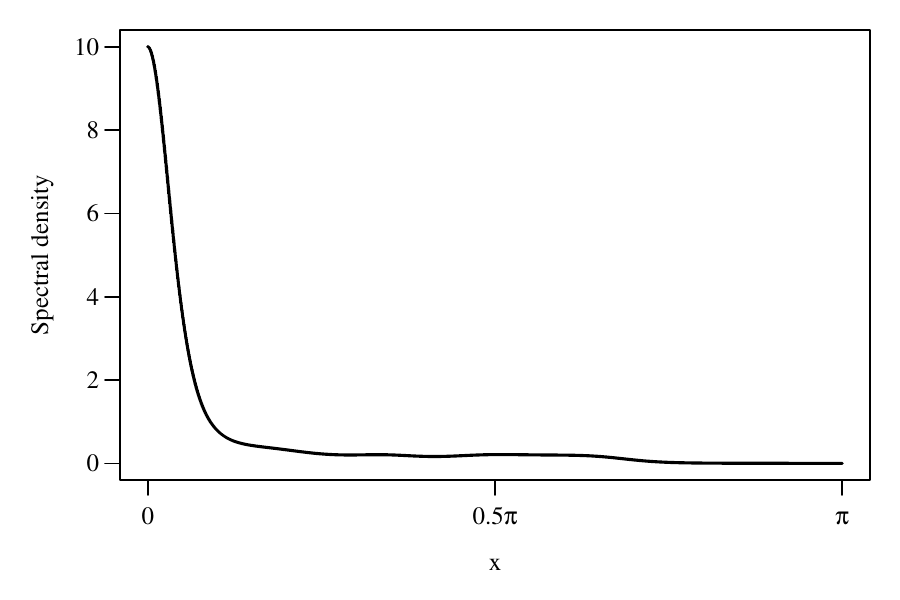}
			\caption{Second frontal tDCS spectral density}
			\label{fig:sdf4}
		\end{subfigure}
		
		\caption{In the first column, difference of potential (in \si{\micro\volt}) registered by the channel E159-E160 along time windows (in \si{\second}) of rest, first frontal tDCS, posterior tDCS, and second frontal tDCS. In the second column, the corresponding estimated spectral density in the interval $[0,1]$.}
		\label{fig:eeg-signals}
	\end{figure}
	
	For each test, performed with our package  \texttt{sdf.test}, we set the number of bins to $T=2100$, as the time series showed approximately Gaussian behaviour after the data transformation, and the significance level to $\alpha=0.05$. The spline degree is set to $q=4$ for all tests, as higher values of $q$ yielded nearly the same spectral density estimators. The smoothing parameter is selected by the maximum likelihood. The $p$-values of the tests are reported in Table \ref{tab:pvalues}.
	
	\begin{table}[h!]
		\centering
		\caption{Table of $p$-values of our test applied to the EEG data}
		\label{tab:pvalues}
		\begin{tabular}{r|cccc}
			& Initial rest & First frontal tDCS & Second frontal tDCS & Posterior tDCS \\ 
			\hline
			Initial rest       & $1$               & $<2.2 \cdot 10^{-16}$ & $<2.2 \cdot 10^{-16}$ & $<2.2 \cdot 10^{-16}$ \\ 
			First frontal tDCS & -                 & $1$                   & $0.6227$               & $<2.2 \cdot 10^{-16}$ \\ 
			Second frontal tDCS& -                 & -                     & $1$                    & $<2.2 \cdot 10^{-16}$ \\ 
			Posterior tDCS     & -                 & -                     & -                      & $1$                    \\ 
		\end{tabular}
		
		\vspace{0.5em}
		{\footnotesize
			Rows correspond to ${\boldsymbol{X}}_1$, the longest time series, and columns to ${\boldsymbol{X}}_2$, the shortest time series. Each entry reports the $p$-value of the associated test applied to ${\boldsymbol{X}}_1$ and ${\boldsymbol{X}}_2$.
		}
	\end{table}
	
	The test rejects the null hypothesis in all comparisons between the initial rest condition and both posterior and frontal tDCS, indicating a brain response to direct current stimulation. In contrast, the test does not reject the null hypothesis when comparing EEG signals obtained under frontal tDCS, but does so when comparing them with those obtained under posterior tDCS.
	These findings align with the channel's position. Since the channel is located in the posterior region of the brain, neurons in this area may be more affected by a posterior stimulation, leading to changes in the corresponding spectral density.
	
	\section{Discussion}
	
	We have demonstrated that our proposed method can detect differences in the stochastic structures of two time series of different lengths -- a problem that often arises in practice. Moreover, our results pave the way for a new approach to change-point detection in the dependence structure of a time series. Indeed, suppose we observe a linear time series with the $\mathrm{MA}(\infty)$ representation $X_t=\sum_{k=0}^{+\infty} \psi_{k}(t/T) Z_{t-k}$, $t=1,\dots,T$ where $Z_t$ is a white noise and the coefficients $\psi_{k}(t/T)$ change at some unknown time $\tau\in(0,1)$ such that $\psi_{k}(t/T)=\psi_{k}^0\mathbf 1_{\{t/T\le\tau\}}+ \psi_{k}^1\mathbf 1_{\{t/T>\tau\}}$. Then the spectral density of the observed piecewise stationary process can be written as a sum of two spectral densities \citep{Preuss_Dette_2015} and the corresponding estimator as a mixutre of the estimators of each spectral density: $\hat f (x)=(\tau/T) \hat f_1(x)+(1-\tau/T) \hat f_2(x)$.  The proposed two-sample testing method can be directly adapted to the problem of change-point detection at time $\tau$.  
	An adaptation to an unknown change-point location can be obtained by properly modifying the contrast function based on the two spectral density estimators before and after a possible change-point. The same statistic based on the $\mathcal{ L}_ 2 $ distance between two spectral densities can also be applied to the estimation of an unknown change-point. This is a promising direction for further research. 
	
	\section*{Acknowledgement}
	
	The authors thank Professors Jorge Caiado, Nuno Crato, and Linyuan Li for providing the original code implementation of their tests. They are also grateful to Alexis Pomares Pastor for granting access to his database and to Karolina Klockmann for her valuable insights.
	This research was supported by the OeAD, Agency for Education and Internationalisation (project no. FR 11/2024), which enabled the collaboration. The work of Farida Enikeeva was funded by the PHC Amadeus 2024 grant (project no. 50076QJ).
	
	

	\appendix
	
	\section{Appendix}
	
	\subsection{Proof of Theorem \ref{th:MinimaxTest}}\label{apx:ProofTheorem}
	
	This section provides the proof for Theorem \ref{th:MinimaxTest}. In this, and all the following proofs, we indicate with $\P_{0}$, $\E_{0}$ and $\Var_{0}$ ($\P_{1}$, $\E_{1}$ and $\Var_{1}$) the probabilities, expectations and variances under the hypothesis $\Hyp_0$ ($\Hyp_1$). Whenever the hypothesis is irrelevant, we drop the index.
	Moreover, we indicate with $C_1, C_2, \dots$ etc. generic constants independent of $n_1$ and $n_2$.
	
	\begin{proof}[of Theorem \ref{th:MinimaxTest}]
		The proof is divided into the following seven steps.
		
		\step{1}
		\textit{Data decomposition and additional notation.}
			
			Given $\mathbf{X}_k \sim \mathcal{N}_{n_k}(0_{n_k}, \Sigma_k)$, we transform the time series following the described procedure, obtaining new observations $Y_k^*$ of equal sample size $\widetilde{T} =2T-2$. Set now $K_k \in \mathbb{R}^{\widetilde{T} \times \widetilde{T}}$ the kernel matrices defined by
			\[
			(K_k)_{t,s} = \frac{1}{\widetilde{T}h} \mathcal{K}_{h_k,q} (x_t, x_s).
			\]
			Some properties of the kernel matrix can be found in the Supplementary Material. In particular, $K_k$ are real, symmetric and circulant matrices, with eigenvalues in $(0,1]$. Moreover, they are simultaneously diagonalizable, and we denote by $K_1 = U \Lambda_1 U^\T$ and $K_2 = U \Lambda_2 U^\T$ their respective eigenvalue decomposition. We indicate their corresponding eigenvalues as $\lambda_{k,t}$ for $t=1, \dots \widetilde{T}$.
			It follows that our test statistic can be written in terms of the kernel matrices as
			\[
			\begin{split}
				S(\mathbf{X}_1,\mathbf{X}_2)
				& = \frac{1}{\widetilde{T}} \norm{ K_1 \mathbf{Y}_1^* - K_2 \mathbf{Y}_2^* }_{\ell_2}^2. 
			\end{split}
			\]
			The vectors $\mathbf{Y}^*_k$ have components that are dependent of each other. However, Lemma \ref{lem:DataDecomposition} provides a decomposition of $\mathbf{Y}^*_k$ into two vectors, 
			\[
			\mathbf{Y}_k^* = \widetilde{\mathbf{Y}}_k^* + \overline{\mathbf{Y}}_k,
			\]
			among which $\widetilde{\mathbf{Y}}_k^*$ has independent components and $\overline{\mathbf{Y}}_k$ has small asymptotic order in probability as the sample size goes to infinity, precisely $\lnorm{\overline{\mathbf{Y}}_k}_{\ell_2} = \mathcal{O}_p (m_k^{-1/2})$ as $n_k\to\infty$. In particular, 
			\[
			\widetilde{Y}_{k,t}^* = \frac{1}{2^{1/2}} \log f_k(x_t) + \epsilon_{k,t} + \zeta_{k,t} + \xi_{k,t}, \quad t=1,\dots,\widetilde{T},
			\]
			with $\labs{\epsilon_{k,t}} \leq c'_k (n_k^{-1} + n_k^{-\eta_k} \log n_k)$ small deterministic error with $c'_k>0$ constant, $\zeta_{k,t} \stackrel{\mathclap{iid}}{\sim} \mathcal{N}(0,m_k^{-1})$, $m_k=n_k/T$, standard Gaussian random variables, and $\xi_{k,t}$ centered independent random variables with decaying moments.
			For compactness, we set $g_{k,t}=2^{-1/2} \log f_k(x_t)$ and define the following vectors
			\begin{align*}
				\mathbf{g}_k & = (g_{k,1}, \dots, g_{k,\widetilde{T}})^\T
				& \boldsymbol{\epsilon}_k & = (\epsilon_{k,1}, \dots, \epsilon_{k,\widetilde{T}})^\T \\
				\boldsymbol{\zeta}_k & = (\zeta_{k,1}, \dots, \zeta_{k,\widetilde{T}})^\T 
				& \boldsymbol{\xi}_k & = (\xi_{k,1}, \dots, \xi_{k,\widetilde{T}})^\T .
			\end{align*}

		\step{2}
		\textit{Proof of part (i).}
			
			We prove that, for $H_\alpha$ as in \eqref{eq:OrderCritValue}, $\psi_{n_1,n_2}$ is asymptotically a test of level $\alpha$ for problem \eqref{eq:Problem}. First, for some small $r_{\alpha}>0$, we bound the type I error using the triangle inequality and the union bound as follows
			\[
			\begin{split}
				\alpha(\psi_{n_1,n_2})
				& = \P_{0} \big\{ S(\mathbf{X}_1,\mathbf{X}_2) > H_\alpha \big\} \\
				& \leq \P_{0} \bigg( \frac{1}{\widetilde{T}^{1/2}} \norm{K_1 \widetilde{\mathbf{Y}}_1^* - K_2 \widetilde{\mathbf{Y}}_2^*}_{\ell_2} + \frac{1}{\widetilde{T}^{1/2}} \norm{ K_1 \overline{\mathbf{Y}}_1 - K_2 \overline{\mathbf{Y}}_2 }_{\ell_2} > H_\alpha^{1/2} \bigg) \\
				& \leq \P_{0} \bigg( \frac{1}{\widetilde{T}^{1/2}} \norm{K_1 \widetilde{\mathbf{Y}}_1^* - K_2 \widetilde{\mathbf{Y}}_2^*}_{\ell_2} + r_{\alpha} > H_\alpha^{1/2} \bigg) + \P_{0} \bigg( \frac{1}{\widetilde{T}^{1/2}} \norm{ K_1 \overline{\mathbf{Y}}_1 - K_2 \overline{\mathbf{Y}}_2 }_{\ell_2} > r_{\alpha} \bigg) .
			\end{split}
			\]
			Our goal is to bound each term asymptotically by $\alpha/2$.
			
			For the second term, once again using the union bound and the properties of the kernel matrices, we can write the following upper bound
			\[
			\begin{split}
				\P_{0} \bigg( \norm{ K_1 \overline{\mathbf{Y}}_1 - K_2 \overline{\mathbf{Y}}_2 }_{\ell_2} > \widetilde{T}^{1/2} r_{\alpha} \bigg) 
				& \leq \P_{0} \bigg( \norm{ \overline{\mathbf{Y}}_1 }_{\ell_2} + \norm{\overline{\mathbf{Y}}_2 }_{\ell_2} > \widetilde{T}^{1/2} r_{\alpha} \bigg) .
			\end{split}
			\]
			Since we know that $\lnorm{\overline{\mathbf{Y}}_k}_{\ell_2} = \mathcal{O}_p (m_k^{-1/2})$, then $\lnorm{ \overline{\mathbf{Y}}_1 }_{\ell_2} + \lnorm{\overline{\mathbf{Y}}_2 }_{\ell_2} = \mathcal{O}_p( \min\{m_1,m_2\}^{-1/2})$, and for any $\sigma >0$ there exists $M_{\sigma}>0$ and $N_{\sigma}>0$ such that we have
			\[
			\P_{0} \Big\{ \Big( \norm{ \overline{\mathbf{Y}}_1 }_{\ell_2} + \norm{\overline{\mathbf{Y}}_2 }_{\ell_2}  \Big) \min\{ m_1, m_2 \}^{1/2} > M_{\sigma} \Big\} \leq \sigma,
			\]
			for all $m_1, m_2 > N_{\sigma}$. Therefore, if we set $r_{\alpha} = M_{\alpha/2} \min \{n_1, n_2\}^{-1/2}$, then 
			\[
			\P_{0} \bigg( \norm{ K_1 \overline{\mathbf{Y}}_1 - K_2 \overline{\mathbf{Y}}_2 }_{\ell_2} > \widetilde{T}^{1/2} r_{\alpha} \bigg) \leq \frac{\alpha}{2} ,
			\]
			as $m_1, m_2 \to \infty$.
			
			For the remaining term, set for simplicity
			\[
			V = \frac{1}{\widetilde{T}} \norm{K_1 \widetilde{\mathbf{Y}}_1^* - K_2 \widetilde{\mathbf{Y}}_2^*}_{\ell_2}^2,
			\]
			so that we can write the following upper bound using Chebyshev's inequality
			\[
			\begin{split}
				\P_{0} \bigg\{ \frac{1}{\widetilde{T}} \norm{K_1 \widetilde{\mathbf{Y}}_1^* - K_2 \widetilde{\mathbf{Y}}_2^*}_{\ell_2}^2 > (H_\alpha^{1/2}-r_{\alpha})^2 \bigg\}
				& \leq \frac{\Var_{0} (V)}{ \big\{ (H_\alpha^{1/2} - r_{\alpha})^2 - \E_{0} (V) \big\}^2}.
			\end{split}
			\]
			This term is bounded by $\alpha/2$ if
			\[
			(H_\alpha^{1/2}-r_{\alpha})^2 \geq \bigg\{ \frac{2 \Var_{0} (V) }{\alpha} \bigg\}^{1/2} + \E_{0} (V),
			\]
			which holds true for $H_\alpha \geq r_\alpha$ defined as follows
			\begin{equation}\label{eq:DefCritValue}
				H_\alpha = \bigg( \bigg[ \bigg\{ \frac{2 \Var_{0} (V) }{\alpha} \bigg\}^{1/2} + \E_0 (V) \bigg]^{1/2} + r_{\alpha} \bigg)^2.
			\end{equation}
			We want to show that $H_\alpha$ satisfies \eqref{eq:OrderCritValue} when $h_k \to 0$ and $h_k T \to \infty$. 
			We introduce the following notation
			\begin{equation}\label{eq:DefMuj}
				\boldsymbol{\mu}_j = K_1 (\mathbf{g}_1 + \boldsymbol{\epsilon}_1) - K_2 (\mathbf{g}_2 + \boldsymbol{\epsilon}_2),
			\end{equation}
			under hypothesis $\Hyp_j$ with $j=0,1$, and
			\begin{equation}\label{eq:Defr}
				\mathbf{v} = K_1 (\boldsymbol{\zeta}_1 + \boldsymbol{\xi}_1) - K_2 (\boldsymbol{\zeta}_2 +\boldsymbol{\xi}_2),
			\end{equation}
			so that we can write $V = \lnorm{\boldsymbol{\mu}_j + \mathbf{v} }_{\ell_2}^2 / \widetilde{T}$, obtaining
			\begin{equation}\label{eq:DecompE0[V]}
				\E_{0} (V) = \frac{1}{\widetilde{T}} \boldsymbol{\mu}_0^\T \boldsymbol{\mu}_0 + \frac{1}{\widetilde{T}} \tr(\Omega) ,
			\end{equation}
			having set $\Omega = \E (\mathbf{v} \mathbf{v}^\T)$.
			Similarly, we can write
			\[
			\begin{split}
				\Var_{0} (V) 
				& = \frac{1}{\widetilde{T}^2} \Big[ 4 \E \big\{ (\boldsymbol{\mu}_0^\T \mathbf{v})^2 \big\} + \Var ( \mathbf{v}^\T \mathbf{v} ) + 4 \E ( \boldsymbol{\mu}_0^\T \mathbf{v} \mathbf{v}^\T \mathbf{v} ) \Big] .
			\end{split}
			\]
			Moreover, using $ (x+y)^{1/2} \leq x^{1/2} + y^{1/2}$ and $2 (xy)^{1/2} \leq x + y$ we get
			\begin{align}
				\bigg\{ \frac{2 \Var_{0} (V) }{\alpha} \bigg\}^{1/2}
				& \leq \frac{2}{\widetilde{T}} \bigg( \frac{2}{\alpha} \Big[ \E \big\{ (\boldsymbol{\mu}_0^\T \mathbf{v})^2 \big\} + \E ( \boldsymbol{\mu}_0^\T \mathbf{v} \mathbf{v}^\T \mathbf{v} ) \Big] \bigg)^{1/2} + \frac{1}{\widetilde{T}} \bigg\{ \frac{2 \Var ( \mathbf{v}^\T \mathbf{v} )}{ \alpha } \bigg\}^{1/2} \notag \\
				& \leq \frac{1}{4} \frac{\E \big\{ (\boldsymbol{\mu}_0^\T \mathbf{v})^2 \big\}}{\widetilde{T} \norm{\Omega}_2} + \frac{1}{4} \frac{\E ( \boldsymbol{\mu}_0^\T \mathbf{v} \mathbf{v}^\T \mathbf{v} )}{\widetilde{T} \norm{\Omega}_2} + \frac{8}{\alpha}  \frac{\norm{\Omega}_2}{\widetilde{T}} + \frac{1}{\widetilde{T}} \bigg\{ \frac{2 \Var ( \mathbf{v}^\T \mathbf{v} )}{ \alpha } \bigg\}^{1/2}. \label{eq:DecompVar0[V]}
			\end{align}
			The analysis of the asymptotic order of each term in \eqref{eq:DecompE0[V]} and \eqref{eq:DecompVar0[V]} can be found in Lemma \ref{lem:Orders} in the Supplementary Material, by which we can deduce that there exists a constant $\mathcal{C}_1 = \mathcal{C}_1 (\eta_1, \eta_2, q, c, M_0, M_1)>0$ such that as $h_k \to 0$, and $Th_k \to \infty$
			\[
			\begin{split}
				H_\alpha & \leq \mathcal{C}_1 \sum_{k=1,2}  \bigg\{ h_k^{2\min\{\eta, 2q\}} +  \frac{h_k^{\min\{\eta, 2q\}-1}}{m_k n_k} +\frac{h_k^{\min\{\eta, 2q\}}}{ n_k^{1/2} } + \frac{h_k^{(\min\{\eta, 2q\} - 1)/2}}{ m_k^{1/2} n_k } \\
				&\qquad +\frac{1}{\alpha n_k}  + \frac{1}{(\alpha h_k)^{1/2}n_k} +  \frac{1}{(\alpha h_k)^{1/4}n_k}
				+ \frac{1}{ \alpha^{1/2} n_k }\\
				&\qquad + \frac{1}{n_k} +\frac{1}{n_k h_k^{1/2}} +\frac{1}{n_k h_k} + \frac{\log^2 n_k}{n_k^{2\eta}} 
				+ \frac{\log n_k}{n_k^{\eta + 1/2}}  \bigg\},
			\end{split}
			\]
			
			which reduces to \eqref{eq:OrderCritValue} for $\eta \geq 1/2$.
		
		\step{3}
		\textit{Proof of part (ii), type II error.}
			
			In this step, we find an upper bound on the Type II error.
			Similarly to what has been done in the previous step, we use the reverse triangle inequality and the union bound as follows
			\[
			\begin{split}
				\beta(\psi_{n_1,n_2})
				& = \P_{1} \big\{ S(\mathbf{X}_1, \mathbf{X}_2) < H_\alpha \big\} \\
				& \leq \P_{1} \bigg( \frac{1}{\widetilde{T}^{1/2}} \abs{ \norm{K_1 \widetilde{\mathbf{Y}}_1^* - K_2 \widetilde{\mathbf{Y}}_2^*}_{\ell_2} - \norm{K_1 \overline{\mathbf{Y}}_1 - K_2 \overline{\mathbf{Y}}_2 }_{\ell_2}} < H_\alpha^{1/2} \bigg) \\
				& \leq \P_{1} \bigg( \frac{1}{\widetilde{T}^{1/2}} \norm{K_1 \widetilde{\mathbf{Y}}_1^* - K_2 \widetilde{\mathbf{Y}}_2^*} - r_{\beta} < H_\alpha^{1/2} \bigg) + \P_{1} \bigg( \norm{ K_1 \overline{\mathbf{Y}}_1 - K_2 \overline{\mathbf{Y}}_2 }_{\ell_2} > \widetilde{T}^{1/2} r_{\beta} \bigg) .
			\end{split}
			\]
			
			Once again, using the asymptotic order in probability of $\lnorm{\overline{\mbox{Y}}_k}_2$, we set $r_{\beta} = M_{\beta/2} \min \{n_1, n_2\}^{-1/2}$, obtaining
			\[
			\P_{1} \bigg( \norm{ K_1 \overline{\mathbf{Y}}_1 - K_2 \overline{\mathbf{Y}}_2 }_{\ell_2} > \widetilde{T}^{1/2} r_{\beta} \bigg) \leq \frac{\beta}{2},
			\]
			as $m_1, m_2 \to \infty$.
			
			We proceed similarly also for the first term, for which we get, using Chebyshev's inequality
			\[
			\begin{split}
				\P_{1} \bigg\{ \frac{1}{\widetilde{T}} \norm{K_1 \widetilde{\mathbf{Y}}_1^* - K_2 \widetilde{\mathbf{Y}}_2^*}_{\ell_2}^2 < (H_\alpha^{1/2} + r_{\beta})^2 \bigg\}
				& \leq \frac{\Var_{1} (V)}{\big\{ ( H_\alpha^{1/2} + r_{\beta})^2 - \E_{1} (V) \big\}^2}.
			\end{split}
			\]
			This term is therefore bounded by $\beta/2$ if
			\begin{equation}\label{eq:TypeIIErrorRequest}
				\E_{1} (V) \geq \bigg\{ \frac{2 \Var_{1} (V)}{\beta} \bigg\}^{1/2} + ( H_\alpha^{1/2} + r_{\beta})^2.
			\end{equation}
		
		\step{4}
		\textit{Proof of part (ii), separation request.}
			
			Recalling \eqref{eq:DefCritValue}, in order for \eqref{eq:TypeIIErrorRequest} to be satisfied, the minimum separation rate $\rho_{n_1,n_2}^*$ of the spectral densities need to verify the following condition
			\begin{equation}\label{eq:RequestI}
				\E_{1} (V) \geq \bigg\{ \frac{2 \Var_{1} (V)}{\beta} \bigg\}^{1/2} +  \bigg( \bigg[ \bigg\{ \frac{2 \Var_{0} (V)}{\alpha} \bigg\}^{1/2} + \E_{0} (V) \bigg]^{1/2} + r_{\alpha} + r_{\beta} \bigg)^2.
			\end{equation}
			
			Similarly to \eqref{eq:DecompE0[V]} and \eqref{eq:DecompVar0[V]}, we can write
			\begin{equation}\label{eq:DecompE1[V]}
				\E_{1} (V) = \frac{1}{\widetilde{T}} \boldsymbol{\mu}_1^\T \boldsymbol{\mu}_1 + \frac{1}{\widetilde{T}} \tr(\Omega) ,
			\end{equation}
			and
			\begin{equation}\label{eq:DecompVar1[V]}
				\bigg\{ \frac{2 \Var_{1} (V)}{\beta} \bigg\}^{1/2}
				\leq 
				\frac{1}{4} \frac{\E \big\{ (\boldsymbol{\mu}_1^\T \mathbf{v})^2 \big\}}{\widetilde{T} \norm{\Omega}_2} + \frac{1}{4} \frac{\E ( \boldsymbol{\mu}_1^\T \mathbf{v} \mathbf{v}^\T \mathbf{v} )}{\widetilde{T} \norm{\Omega}_2} + \frac{8}{\beta} \frac{\norm{\Omega}_2}{\widetilde{T}} + \frac{1}{\widetilde{T}} \bigg\{ \frac{2 \Var ( \mathbf{v}^\T \mathbf{v} )}{ \beta } \bigg\}^{1/2} .
			\end{equation}
			
			By \eqref{eq:DecompE0[V]}, \eqref{eq:DecompVar0[V]}, \eqref{eq:DecompE1[V]} and \eqref{eq:DecompVar1[V]}, the request in \eqref{eq:RequestI} is satisfied when the following condition holds
			\begin{align}\label{eq:RequestII}
				\frac{1}{\widetilde{T}} \boldsymbol{\mu}_1^\T \boldsymbol{\mu}_1 
				& \geq \frac{1}{\widetilde{T}} \boldsymbol{\mu}_0^\T \boldsymbol{\mu}_0 + \frac{\E \big\{ (\boldsymbol{\mu}_0^\T \mathbf{v})^2 \big\}}{4 \widetilde{T} \norm{\Omega}_2} + \frac{\E \big\{ (\boldsymbol{\mu}_1^\T \mathbf{v})^2 \big\}}{4 \widetilde{T} \norm{\Omega}_2} + \frac{\E ( \boldsymbol{\mu}_0^\T \mathbf{v} \mathbf{v}^\T \mathbf{v} )}{4 \widetilde{T} \norm{\Omega}_2} + \frac{\E ( \boldsymbol{\mu}_1^\T \mathbf{v} \mathbf{v}^\T \mathbf{v} )}{4 \widetilde{T} \norm{\Omega}_2} + \notag \\
				& \quad + 8 \bigg( \frac{1}{\alpha} + \frac{1}{\beta} \bigg) \frac{\norm{\Omega}_2}{ \widetilde{T} } + 2^{1/2} \bigg( \frac{1}{ \alpha^{1/2} } + \frac{1}{ \beta^{1/2} } \bigg) \frac{ \{ \Var ( \mathbf{v}^\T \mathbf{v} ) \}^{1/2}}{\widetilde{T}} + (r_{\alpha} + r_{\beta})^2 +  (r_{\alpha} + r_{\beta}) \times \notag \\
				& \quad \times \Bigg[ \frac{\E \big\{ (\boldsymbol{\mu}_0^\T \mathbf{v})^2 \big\}}{\widetilde{T} \norm{\Omega}_2} + \frac{\E ( \boldsymbol{\mu}_0^\T \mathbf{v} \mathbf{v}^\T \mathbf{v} )}{\widetilde{T} \norm{\Omega}_2} + \frac{32}{\alpha} \frac{\norm{\Omega}_2}{\widetilde{T}} + 4 \bigg\{ \frac{2 \Var_{0} (V)}{\alpha} \bigg\}^{1/2} + \frac{4}{\widetilde{T}} \boldsymbol{\mu}_0^\T \boldsymbol{\mu}_0 + \frac{4}{\widetilde{T}} \tr(\Omega) \Bigg]^{1/2}.
			\end{align}
			
			Once again, the orders of each of the terms can be found in Lemma \ref{lem:Orders} in the Supplementary Material. Recalling that $\eta \geq 1/2$, we can reformulate \eqref{eq:RequestII} as
			\begin{align}\label{eq:RequestIII}
				& \frac{1}{\widetilde{T}} \boldsymbol{\mu}_1^\T \boldsymbol{\mu}_1 - \frac{\boldsymbol{\mu}_1^\T \Omega \boldsymbol{\mu}_1}{4 \widetilde{T} \norm{\Omega}_2} - \bigg(\frac{1}{n_1 m_1 h_1} + \frac{1}{n_2 m_2 h_2}\bigg) \frac{\norm{\boldsymbol{\mu}_1}_2}{\widetilde{T}^{1/2}} \notag \\
				& \qquad \geq C_1 \Bigg\{ h_1^{2\min\{\eta, 2q\}} + h_2^{2\min\{ \eta, 2q \}} + \bigg( \frac{1}{ \alpha^{1/2} } + \frac{1}{ \beta^{1/2} } \bigg) \bigg( \frac{1}{n_1 h_1^{1/2}} + \frac{1}{n_2 h_2^{1/2}} \bigg) \Bigg\} .
			\end{align}
			
			The last term of the left-hand side can be bounded using the relation $2 x y \leq x^2 + y^2$, from which we get
			\[
			\frac{1}{n_k m_k h_k} \frac{\norm{\boldsymbol{\mu}_1}_2}{\widetilde{T}^{1/2}} \leq \frac{\norm{\boldsymbol{\mu}_1}_2^2}{4 \widetilde{T}} + \frac{1}{(n_k m_k h_k)^2},
			\]
			and as a consequence the request \eqref{eq:RequestIII} is satisfied whenever the following holds
			\begin{equation}\label{eq:RequestIV}
				\frac{1}{\widetilde{T}} \boldsymbol{\mu}_1^\T \boldsymbol{\mu}_1 - \frac{\boldsymbol{\mu}_1^\T \Omega \boldsymbol{\mu}_1}{ 2 \widetilde{T} \norm{\Omega}_2} \geq C_2 \Bigg\{ h_1^{2\min\{\eta, 2q\}} + h_2^{2\min\{ \eta, 2q \}} + \bigg( \frac{1}{ \alpha^{1/2} } + \frac{1}{ \beta^{1/2} } \bigg) \bigg( \frac{1}{n_1 h_1^{1/2}} + \frac{1}{n_2 h_2^{1/2}} \bigg) \Bigg\}.
			\end{equation}
			
			Set $\tilde{\mathbf{g}}_k = U^\T \mathbf{g}_k$ and $\tilde{\boldsymbol{\epsilon}}_k = U^\T \boldsymbol{\epsilon}_k$ the Fourier transform of $\mathbf{g}_k$ and $\boldsymbol{\epsilon}_k$ respectively, and indicate by $\mathbf{I}_{\widetilde{T}}$ the $\widetilde{T} \times \widetilde{T}$ identity matrix. The left-hand side of \eqref{eq:RequestIV} can be written as
			\[
			\begin{split}
				\frac{1}{\widetilde{T}} \boldsymbol{\mu}_1^\T \bigg(\mathbf{I}_{\widetilde{T}} - \frac{\Omega}{2 \norm{\Omega}_2}\bigg) \boldsymbol{\mu}_1 
				& = \frac{1}{\widetilde{T}} \sum_{t=1}^{\widetilde{T}} \bigg(1 - \frac{\lambda_t^{(\Omega)}}{2 \lambda_1^{(\Omega)}} \bigg) (\lambda_{1,t} \tilde{g}_{1,t} - \lambda_{2,t} \tilde{g}_{2,t} + \lambda_{1,t} \tilde{\epsilon}_{1,t} - \lambda_{2,t} \tilde{\epsilon}_{2,t} )^2 \\
				& = \frac{1}{\widetilde{T}} (\tilde{\mathbf{g}}_1^\T, - \tilde{\mathbf{g}}_2^\T) \begin{pmatrix}
					\Lambda_1 \\ \Lambda_2
				\end{pmatrix} \bigg(\mathbf{I}_{\widetilde{T}} - \frac{\Lambda^{(\Omega)}}{2\norm{\Omega}_2}\bigg) (\Lambda_1, \Lambda_2) \begin{pmatrix}
					\tilde{\mathbf{g}}_1 \\ -\tilde{\mathbf{g}}_2
				\end{pmatrix} + \mathcal{O} \bigg( \frac{1}{n_1} + \frac{1}{n_2} \bigg) \\
				& = \frac{1}{\widetilde{T}} (\tilde{\mathbf{g}}_1^\T, - \tilde{\mathbf{g}}_2^\T) \Bigg\{ \begin{pmatrix}
					\Lambda_1^2 & \Lambda_1 \Lambda_2 \\ 
					\Lambda_2 \Lambda_1 & \Lambda_2^2
				\end{pmatrix} - \frac{1}{2\norm{\Omega}_2} \begin{pmatrix}
					\Lambda_1 \Lambda^{(\Omega)} \Lambda_1 & \Lambda_1 \Lambda^{(\Omega)} \Lambda_2 \\
					\Lambda_2 \Lambda^{(\Omega)} \Lambda_1 & \Lambda_2 \Lambda^{(\Omega)} \Lambda_2
				\end{pmatrix} 
				\Bigg\}
				\begin{pmatrix}
					\tilde{\mathbf{g}}_1 \\ -\tilde{\mathbf{g}}_2
				\end{pmatrix} + \\
				& \quad + \mathcal{O} \bigg( \frac{1}{n_1} + \frac{1}{n_2} \bigg).
			\end{split}
			\]
			Defining the following $2\widetilde{T} \times 2\widetilde{T}$ matrix 
			\[
			A = \begin{pmatrix}
				\mathbf{I}_{\widetilde{T}} - 2 \Lambda_1^2 + \frac{1}{ \norm{\Omega} } \Lambda_1^2 \Lambda^{(\Omega)} & \mathbf{I}_{\widetilde{T}} - 2 \Lambda_1 \Lambda_2 + \frac{1}{ \norm{\Omega} } \Lambda_1 \Lambda_2 \Lambda^{(\Omega)} \\
				\mathbf{I}_{\widetilde{T}} - 2 \Lambda_1 \Lambda_2 + \frac{1}{ \norm{\Omega} } \Lambda_1 \Lambda_2 \Lambda^{(\Omega)} & \mathbf{I}_{\widetilde{T}} - 2 \Lambda_2^2 + \frac{1}{ \norm{\Omega} } \Lambda_2^2 \Lambda^{(\Omega)}
			\end{pmatrix} ,
			\]
			we can write
			\begin{align}\label{eq:LHS}
				\frac{1}{\widetilde{T}} \boldsymbol{\mu}_1^\T \bigg(\mathbf{I}_{\widetilde{T}} - \frac{\Omega}{2\norm{\Omega}_2}\bigg) \boldsymbol{\mu}_1
				& = \frac{1}{2 \widetilde{T}} \norm{\tilde{\mathbf{g}}_1 - \tilde{\mathbf{g}}_2}_{\ell_2}^2 - \frac{1}{2 \widetilde{T}} (\tilde{\mathbf{g}}_1^\T, - \tilde{\mathbf{g}}_2^\T) A \begin{pmatrix}
					\tilde{\mathbf{g}}_1 \\ -\tilde{\mathbf{g}}_2
				\end{pmatrix}
				+ \mathcal{O} \bigg( \frac{1}{n_1} + \frac{1}{n_2} \bigg).
			\end{align} 
			In the next step we prove that
			\begin{equation}\label{eq:OrderAB}
				\frac{1}{2 \widetilde{T}} (\tilde{\mathbf{g}}_1^\T, - \tilde{\mathbf{g}}_2^\T) A \begin{pmatrix}
					\tilde{\mathbf{g}}_1 \\ -\tilde{\mathbf{g}}_2
				\end{pmatrix} = \mathcal{O} \big( \max\{h_1, h_2\}^{\eta^*} \big) ,
			\end{equation}
			as $h_k \to 0$ and $h_k T \to \infty$, from which, recalling \eqref{eq:LHS}, the request \eqref{eq:RequestIV} becomes
			\begin{equation}\label{eq:RequestV}
				\frac{1}{\widetilde{T}} \norm{\mathbf{g}_1-\mathbf{g}_2}_{\ell_2}^2 
				\geq C_3 \Bigg\{ h_1^{2\eta^*} + h_2^{2\eta^*} + \bigg( \frac{1}{\alpha^{1/2}} + \frac{1}{\beta^{1/2}} \bigg) \bigg( \frac{1}{n_1 h_1^{1/2}} + \frac{1}{n_2 h_2^{1/2}} \bigg) \Bigg\}.
			\end{equation}
		
		\step{5}
		\textit{Proof of part (ii), proof of \eqref{eq:OrderAB}.}
			
			We define the $2\widetilde{T} \times 2\widetilde{T}$ matrix
			\[
			B = \begin{pmatrix}
				(\mathbf{I}_{\widetilde{T}} - \Lambda_1^2)^2 & (\mathbf{I}_{\widetilde{T}} - \Lambda_1 \Lambda_2)^2 \\
				(\mathbf{I}_{\widetilde{T}} - \Lambda_1 \Lambda_2)^2 & (\mathbf{I}_{\widetilde{T}} - \Lambda_2^2)^2
			\end{pmatrix}.
			\]
			We split the proof into two steps.
			
			First, we will prove that 
			\begin{equation}\label{eq:OrderB}
				\frac{1}{\widetilde{T}} (\tilde{\mathbf{g}}_1^\T, - \tilde{\mathbf{g}}_2^\T) B
				\begin{pmatrix}
					\tilde{\mathbf{g}}_1 \\ -\tilde{\mathbf{g}}_2
				\end{pmatrix} 
				= \mathcal{O} \Big\{ \max\{h_1, h_2\}^{2 \min\{ \eta_1, \eta_2, 2q\}} \Big\}.
			\end{equation}
			By definition of $B$, we have 
			\[
			\begin{split}
				(\tilde{\mathbf{g}}_1^\T, - \tilde{\mathbf{g}}_2^\T) B
				\begin{pmatrix}
					\tilde{\mathbf{g}}_1 \\ -\tilde{\mathbf{g}}_2
				\end{pmatrix} 
				& = \tilde{\mathbf{g}}_1^\T (\mathbf{I}_{\widetilde{T}} - \Lambda_1^2)^2 \tilde{\mathbf{g}}_1 + \tilde{\mathbf{g}}_2^\T (\mathbf{I}_{\widetilde{T}} - \Lambda_2^2)^2 \tilde{\mathbf{g}}_2 - 2  \tilde{\mathbf{g}}_1^\T (\mathbf{I}_{\widetilde{T}} - \Lambda_1 \Lambda_2)^2 \tilde{\mathbf{g}}_2 .
			\end{split}
			\]
			We start by looking at the mixed term. 
			By the properties of the Kernel matrices described in the Supplementary Material, we know that the eigenvalues $\lambda_{2,t}$ are symmetric around $\widetilde{T}/2$ and they have the following explicit form
			\[
			\lambda_{2,t} = \frac{1}{1 + (c_t \pi h_2 t)^{2q}},
			\]
			for some $c_t \in [2^{1-(2q)^{-1}}, 2]$ for all $t = 1, \dots, \widetilde{T}/2$.
			Without loss of generality assume $h_1 \leq h_2$, and consequently $\lambda_{1,t} \geq \lambda_{2,t}$. 
			Since $f_k (x) \in [c, M_0]$ and the logarithm is Lipschitz continuous on a compact, we can state that $\mathbf{g}_k$ is itself $\eta_k$-Hölder continuous, and, therefore, its Fourier coefficients $\tilde{g}_{k,t}$ decay as $t^{-\eta_k}$.
			By symmetry, we can write
			\[
			\begin{split}
				\abs{\frac{1}{\widetilde{T}} \tilde{\mathbf{g}}_1^\T (\mathbf{I}_{\widetilde{T}} - \Lambda_1 \Lambda_2)^2 \tilde{\mathbf{g}}_2 }
				& \leq \frac{2}{\widetilde{T}} \sum_{t=1}^{\widetilde{T}/2} (1 - \lambda_{2,t}^2)^2 \abs{\tilde{g}_{1,t} \tilde{g}_{2,t}} \\
				& \leq \frac{2}{\widetilde{T}} \sum_{t=1}^{\widetilde{T}/2} \frac{4(c_t \pi h_{2} t)^{4q} + 4(c_t \pi h_{2} t)^{6q} + (c_t \pi h_{2} t)^{8q}}{ \{1 + (c_t \pi h_{2} t)^{2q}\}^4} t^{-2 \min\{\eta_1, \eta_2\}} , 
			\end{split}
			\]
			from which if $\min\{\eta_1, \eta_2\} > 2q+1/2$, then
			\[
			\begin{split}
				\frac{8}{\widetilde{T}} \sum_{t=1}^{\widetilde{T}/2} \frac{(c_t \pi h_{2} t)^{4q} }{ \{1 + (c_t \pi h_{2} t)^{2q}\}^4 } t^{-2 \min\{\eta_1, \eta_2\}}
				& \leq 8 (2 \pi)^{4q}\frac{ h_{2}^{4q} }{\widetilde{T}}  \sum_{t=1}^{\widetilde{T}/2} \frac{1}{ t^{2 \min\{\eta_1, \eta_2\}- 4q}} \\	
				& \leq 8 (2 \pi)^{4q}\frac{ h_{2}^{4q} }{\widetilde{T}} \zeta \{2 \min\{\eta_1, \eta_2\} - 4q\} \\
				& = \mathcal{O} \bigg( \frac{ h_{2}^{4q} }{\widetilde{T}} \bigg),
			\end{split}
			\]
			where $\zeta$ is the Riemann zeta function, if $2q < \min\{\eta_1, \eta_2\} \leq 2q+1/2$, then 
			\[
			\begin{split}
				\frac{8}{\widetilde{T}} \sum_{t=1}^{\widetilde{T}/2} \frac{(c_t \pi h_{2} t)^{4q} }{ \{1 + (c_t \pi h_{2} t)^{2q}\}^4} t^{-2 \min\{\eta_1, \eta_2\}}
				& \leq 8 (2 \pi)^{4q}\frac{ h_{2}^{4q} }{\widetilde{T}}  H_{\frac{\widetilde{T}}{2}, 2 \min\{\eta_1, \eta_2\} - 4q} 
				= \mathcal{O} \big( h_{2}^{4q} \big),
			\end{split}
			\]
			where $H_{i,j}$ is the $i$th generalized harmonic number of order $j$, and if $1/2 < \min\{\eta_1, \eta_2\} \leq 2q $, then
			\[
			\begin{split}
				\frac{8}{\widetilde{T}} \sum_{t=1}^{\widetilde{T}/2} \frac{(c_t \pi h_{2} t)^{4q} }{ \{1 + (c_t \pi h_{2} t)^{2q} \}^4 } t^{-2 \min\{\eta_1, \eta_2\}}
				& \leq 8 (2 \pi)^{2 \min\{\eta_1, \eta_2\} }\frac{ h_{2}^{2 \min\{\eta_1, \eta_2\}} }{\widetilde{T}} \sum_{t=1}^{\widetilde{T}/2} \frac{(c_t \pi h_{2} t)^{4q - 2 \min\{\eta_1, \eta_2\}} }{ \{1 + (c_t \pi h_{2} t)^{2q}\}^4 }  \\
				& \leq 4 (2 \pi)^{2 \min\{\eta_1, \eta_2\} } \frac{ h_{2}^{2 \min\{\eta_1, \eta_2\}} }{\widetilde{T}} \widetilde{T} 
				= \mathcal{O} \Big\{ h_{2}^{2 \min\{\eta_1, \eta_2\}} \Big\} ,
			\end{split}
			\]
			from which, since the choice $h_1 \leq h_2$ is arbitrary, we obtain \eqref{eq:OrderB}.
			The remaining terms follow similarly, as
			\[
			\begin{split}
				\frac{1}{\widetilde{T}} \tilde{\mathbf{g}}_k^\T (\mathbf{I}_{\widetilde{T}} - \Lambda_k^2)^2 \tilde{\mathbf{g}}_k 
				& = \frac{2}{\widetilde{T}} \sum_{t=1}^{\widetilde{T}/2} (1 - \lambda_{k,t}^2)^2 \tilde{g}_{k,t}^2  \\
				& = \frac{2}{\widetilde{T}} \sum_{t=1}^{\widetilde{T}/2} \frac{4(c_t \pi h_k t)^{4q} + 4(c_t \pi h_k t)^{6q} + (c_t \pi h_k t)^{8q}}{ \{1 + (c_t \pi h_k t)^{2q} \}^4} \tilde{g}_{k,t}^2 \\
				& = \mathcal{O} \Big( h_k^{2 \min\{ \eta_1, \eta_2, 2q\}} \Big) .
			\end{split}
			\]
			
			As a second step, we now prove that
			\begin{equation}\label{eq:OrderA}
				\frac{1}{\widetilde{T}} (\tilde{\mathbf{g}}_1^\T, - \tilde{\mathbf{g}}_2^\T) (A-B)
				\begin{pmatrix}
					\tilde{\mathbf{g}}_1 \\ -\tilde{\mathbf{g}}_2
				\end{pmatrix} 
				= \mathcal{O} \Big( \max\{h_1, h_2\}^{2 \eta^*} \Big).
			\end{equation}
			Observe that
			\[
			\begin{split}
				(\tilde{\mathbf{g}}_1^\T, - \tilde{\mathbf{g}}_2^\T) (A - B)
				\begin{pmatrix}
					\tilde{\mathbf{g}}_1 \\ -\tilde{\mathbf{g}}_2
				\end{pmatrix}
				& = \tilde{\mathbf{g}}_1^\T \Lambda_1^2 \bigg( \frac{1}{\norm{\Omega}}\Lambda^{(\Omega)} - \Lambda_1^2\bigg) \tilde{\mathbf{g}}_1 + \tilde{\mathbf{g}}_2^\T \Lambda_2^2 \bigg( \frac{1}{\norm{\Omega}} \Lambda^{(\Omega)} - \Lambda_2^2 \bigg) \tilde{\mathbf{g}}_2 \\
				& \quad - 2  \tilde{\mathbf{g}}_1^\T \Lambda_1 \Lambda_2 \bigg(\frac{1}{\norm{\Omega}} \Lambda^{(\Omega)} - \Lambda_1 \Lambda_2 \bigg) \tilde{\mathbf{g}}_2  ,
			\end{split}
			\]
			Now,
			\[
			\begin{split}
				& \frac{1}{\widetilde{T}} \tilde{\mathbf{g}}_1^\T \Lambda_1^2 \bigg( \frac{1}{\norm{\Omega}}\Lambda^{(\Omega)} - \Lambda_1^2\bigg) \tilde{\mathbf{g}}_1 + \frac{1}{\widetilde{T}} \tilde{\mathbf{g}}_2^\T \Lambda_2^2 \bigg( \frac{1}{\norm{\Omega}} \Lambda^{(\Omega)} - \Lambda_2^2 \bigg) \tilde{\mathbf{g}}_2 \\
				& = \frac{2}{\widetilde{T}} \sum_{t=1}^{\widetilde{T}/2} \Bigg\{ \lambda_{1,t}^2 \frac{m_1 (\lambda_{2,t}^2 - \lambda_{1,t}^2)}{m_1 + m_2} \tilde{g}_{1,t}^2 + \lambda_{2,t}^2 \frac{m_2 (\lambda_{1,t}^2 - \lambda_{2,t}^2)}{m_1 + m_2} \tilde{g}_{2,t}^2 \Bigg\} \\
				& = \mathcal{O} \Big( \max\{h_1, h_2\}^{2 \eta^*} \Big) ,
			\end{split}
			\]
			because, without loss of generality we can assume that $h_1 \leq h_2$, and consequently $\lambda_{1,t} \geq \lambda_{2,t}$, making the sum of the first term negative and
			\[
			\begin{split}
				\frac{1}{\widetilde{T}} \sum_{t=1}^{\widetilde{T}} \lambda_{2,t}^2 \frac{m_1 (\lambda_{1,t}^2 - \lambda_{2,t}^2)}{m_1 + m_2} \tilde{g}_{1,t}^2
				& \leq \frac{2}{\widetilde{T}} \sum_{t=1}^{\widetilde{T}/2} \frac{m_1}{m_1 + m_2} \frac{ 2 (c_t \pi t)^{2q} (h_2^{2q} - h_1^{2q}) + (c_t \pi t)^{4q} (h_2^{4q} - h_1^{4q})}{(1+(c_t \pi h_1 t)^{2q})^2 (1 + (c_t \pi h_2 t)^{2q})^2} \tilde{g}_{2,t}^2 \\
				& = \mathcal{O} \Big( h_{2}^{2 \eta^*} \Big).
			\end{split}
			\]
			Similarly, assuming without loss of generality that $h_1 \leq h_2$, and consequently $\lambda_{1,t} \geq \lambda_{2,t}$ for all $t=1,\dots, \widetilde{T}$,
			\[
			\begin{split}
				\abs{ \frac{1}{\widetilde{T}} \tilde{\mathbf{g}}_1^\T \Lambda_1 \Lambda_2 \bigg( \frac{1}{\norm{\Omega}}\Lambda^{(\Omega)} - \Lambda_1 \Lambda_2 \bigg) \tilde{\mathbf{g}}_2 } 
				& = \frac{2}{\widetilde{T}} \abs{ \sum_{t=1}^{\widetilde{T}/2} \lambda_{1,t} \lambda_{2,t}  \bigg(\frac{m_2 \lambda_{1,t}^2 + m_1 \lambda_{2,t}^2}{m_1 + m_2} - \lambda_{1,t} \lambda_{2,t} \bigg) \tilde{g}_{1,t} \tilde{g}_{2,t} } \\
				& \leq \frac{2}{\widetilde{T}} \sum_{t=1}^{\widetilde{T}/2} \lambda_{1,t} \lambda_{2,t} \abs{(\lambda_{1,t} - \lambda_{2,t}) \frac{m_2 \lambda_{1,t} - m_1 \lambda_{2,t} }{m_1 + m_2}} \abs{ \tilde{g}_{1,t} \tilde{g}_{2,t} } \\
				& = \mathcal{O} \Big( h_{2}^{2 \eta^*} \Big),
			\end{split}
			\]
			having
			\[
			\begin{split}
				& \frac{2}{\widetilde{T}} \sum_{t=1}^{\widetilde{T}/2} \lambda_{1,t} \lambda_{2,t} \abs{(\lambda_{1,t} - \lambda_{2,t}) \frac{m_2 \lambda_{1,t} - m_1 \lambda_{2,t} }{m_1 + m_2}} \abs{ \tilde{g}_{1,t} \tilde{g}_{2,t} } \\
				& \leq \frac{2}{\widetilde{T}} \sum_{t=1}^{\widetilde{T}/2} \frac{(c_t \pi t)^{2q} \abs{h_2^{2q} - h_1^{2q}} }{\{1+(c_t \pi t h_1)^{2q}\}\{1 + (c_t \pi t h_2)^{2q}\}} t^{- 2 \min\{\eta_1, \eta_2\}} \\
				& \leq \frac{2}{\widetilde{T}} \sum_{t=1}^{\widetilde{T}/2} \frac{(c_t \pi t)^{2q} h_{2}^{2q} }{\{1+(c_t \pi t h_1)^{2q}\}\{1 + (c_t \pi t h_2)^{2q}\}} t^{- 2 \min\{\eta_1, \eta_2\}} \\
				& = \mathcal{O} \Big( h_{2}^{2 \eta^*} \Big).
			\end{split}
			\]
		
		\step{6}
		\textit{Proof of part (ii), norm of the difference of spectral densities.}
			
			We observe that $(g_1 - g_2)^2 \in \mathcal{F}_{\min\{ \eta_1, \eta_2\}}$, therefore, by Lemma \ref{lem:Sum2Integral} in the Supplementary Material, we can write
			\[
			\begin{split}
				\frac{1}{\widetilde{T}} \sum_{t=1}^{\widetilde{T}} \bigg\{g_1 \bigg( \frac{t}{\widetilde{T}} \bigg) - g_2 \bigg( \frac{t}{\widetilde{T}} \bigg) \bigg\}^2 
				= \norm{g_1 - g_2}_{\mathcal{L}_2([0,1])}^2 + \mathcal{O} (\widetilde{T}^{\min\{ \eta_1, \eta_2\}}) .
			\end{split}
			\]
			Observe that by \eqref{eq:ConditionNuk} we can write $\mathcal{O} (\widetilde{T}^{\min\{ \eta_1, \eta_2\}}) = \mathcal{O}(\max\{h_k\}^{2 \eta^*})$.
			
			Observe also that
			\[
			\norm{g_1 - g_2}_{\mathcal{L}_2 ([0,1])}^2 
			= \frac{1}{2} \int_0^1 \log^2\frac{f_1(x)}{f_2(x)} dx \geq \frac{1}{2} \int_0^1 \frac{\{f_1(x)-f_2(x)\}^2}{\max \{f_1(x)^2, f_2(x)^2\}} dx \geq \frac{\norm{f_1 - f_2}_{\mathcal{L}_2 ([0,1])}^2}{2 M_0^2}.
			\]
			Finally, the request \eqref{eq:RequestV} becomes
			\[
			\norm{f_1 - f_2}_{\mathcal{L}_2 ([0,1])}^2
			\geq C_4 \Bigg\{ h_1^{2\eta^*} + h_2^{2\eta^*} + \bigg( \frac{1}{\alpha^{1/2}} + \frac{1}{\beta^{1/2}} \bigg) \bigg( \frac{1}{n_1 h_1^{1/2}} + \frac{1}{n_2 h_2^{1/2}} \bigg) \Bigg\}.
			\]
			Using the simple inequality $\sqrt{a+b}\le \sqrt{a}+\sqrt{b}$, we can easily see that this relation is satisfied for \eqref{eq:radiusI}. Therefore we have for all $C > \mathcal{C}_2$
			\[
			\limsup_{n_1, n_2 \to \infty} \beta \Big\{ \psi_{n_1,n_2}, \Theta_{n_1,n_2}^1 (C \varphi_{n_1,n_2})\Big\} \leq \beta.
			\]

		\step{7}
		\textit{Proof of part (iii).}
			
			If $h_k = \mathcal{O} \Big(n_k^{- \frac{2}{ 4 \eta^* +1 } } \Big)$, the above condition becomes 
			\[
			\begin{split}
				\norm{f_1 - f_2}_2^2
				& \geq C_5 \Bigg\{ \bigg( \frac{1}{\alpha^{1/2}} + \frac{1}{\beta^{1/2}} \bigg) \bigg( \frac{1}{n_1} + \frac{1}{n_2} \bigg)^{\frac{4 \eta^*}{ 4\eta^* +1}} \Bigg\} , \\
			\end{split}
			\]
			where we have used the fact that $x^{-a} + y^{-a} < 2 (x^{-1} + y^{-1})^{a}$ for $a \in (0,1)$. Then the separation radius satisfies 
			\[
			\rho_{n_1,n_2}^* \leq \mathcal{C}_2  \bigg( \frac{1}{\alpha^{1/2}} + \frac{1}{\beta^{1/2}} \bigg)^{1/2} \bigg( \frac{1}{n_1} + \frac{1}{n_2} \bigg)^{\frac{2 \eta^*}{ 4\eta^* +1}}\{1+o(1)\},
			\]
			and recalling \eqref{eq:OrderCritValue},
			\[
			H_\alpha = \mathcal{C}_3 \Bigg\{ \bigg( \frac{1}{n_1} + \frac{1}{n_2} \bigg)^{\frac{4 \eta^* - 1}{ 4\eta^* +1}} \Bigg\}  \{ 1 + o(1) \},
			\]
			for some constant $\mathcal{C}_3 = \mathcal{C}_3 (\eta_1, \eta_2, q, \alpha,c,M_0,M_1)$.
		
	\end{proof}
	
	\subsection{Proof of Corollary \ref{cor:AdaptiveTest}}\label{apx:ProofCorollary}
	
	\begin{proof}
		The proof of Corollary \ref{cor:AdaptiveTest} follows the same steps of the proof of Theorem \ref{th:MinimaxTest}.
		The only difference appears in the proof of \eqref{eq:OrderAB}, for which instead will hold 
		\[
		\frac{1}{2 \widetilde{T}} (\tilde{\mathbf{g}}_1^\T, - \tilde{\mathbf{g}}_2^\T) A \begin{pmatrix}
			\tilde{\mathbf{g}}_1 \\ -\tilde{\mathbf{g}}_2
		\end{pmatrix} = \mathcal{O} \big( \max\{h_1, h_2\}^{\min\{\eta_1,\eta_2,2q\}} \big) .
		\]
		In fact, since $n_1=n_2=n$, it follows that $T=\floor{n^\nu}$ bins contain the same number $m=n/T$ of observations each. As a consequence, the bandwidths are equal $h_1=h_2=h$, and therefore the kernel matrices are equal as well $K_1=K_2=K$. This means that $\Omega=2 K^2/m$, $\norm{\Omega}=2/m$, and the matrix $A$ coincide with $B$, for which \eqref{eq:OrderB} still holds.		
	\end{proof}
	
	\section{Supplementary Material}
	
	\subsection{Kernel matrix}
	\label{apx:Kernel matrix}
	
	To discuss the proof of Theorem \ref{th:MinimaxTest}, we need to report some general properties of the smoothing spline kernel.
	Let $h \in \mathbb{R}^+$ be a bandwidth and $q \in \mathbb{N} $ a penalty order. Set $K=K(h,q) \in \mathbb{R}^{N \times N}$ the matrix such that
	\[
	K_{t,s} = \frac{1}{Nh} \mathcal{K}_{h,q} (x_t, x_s),
	\]
	for all $t,s =1, \dots , N$.
	An explicit expression of the kernel is given in \citet{schwarz2016unified}, where it is proven that there exist $\kappa_1, \dots, \kappa_N \in \mathbb{R}$ coefficients, with $\kappa_{N-j}=\kappa_j$ for $j=1,\dots,\floor{N/2}$, such that the components of the kernel matrix $K$ can be written as
	\[
	K_{t,s} = \frac{1}{N} \sum_{j=1}^{N} \frac{\cos \{ 2 \pi j (t-s)/N \} }{1+ h^{2q} \kappa_j},
	\]
	and
	\begin{equation}\label{eq:DefKappaj}
		\kappa_j = \frac{(2 \pi j)^{2q} \, \mbox{sinc} (\pi j / N)^{2q}}{Q_{2q-2}(j/N)},
	\end{equation}
	with $Q_{2q-2}(z)$ polynomial of the form
	\[
	Q_{2q-2} (z) = \sum_{l=-\infty}^{\infty} \mbox{sinc} \{ \pi (z+l) \}^{2q}.
	\]
	It is easy to see that $K$ is a real, symmetric and circulant matrix; therefore its eigenvalue decomposition $K=U \Lambda U^\T$ can be easily derived, and the matrix $U$ does not depend on $h$ and $q$ \citep[see][]{brockwell1991time}. In particular, we consider $\Lambda = \mbox{diag}(\lambda_1,\dots,\lambda_N)$ with ordered eigenvalues such that for $j=1,\dots,N$ the $j$th eigenvalue is
	\begin{equation}\label{eq:DefLambdaj}
		\lambda_j = \frac{1}{1+ h^{2q} \kappa_j}.
	\end{equation}
	As a direct consequence of this sorting, it holds that $\lambda_{N-j}=\lambda_j$ for $j=1,\dots, \floor{N/2}$, therefore if $N$ is odd, $\lambda_N$ has multiplicity one, and the remaining $\floor{N/2}$ eigenvalues have multiplicity 2, and if $N$ is even, both $\lambda_N$ and $\lambda_{N/2}$ have multiplicity one and the remaining $N-2$ eigenvalues have multiplicity 2.
	It can be easily shown that $0 <\lambda_j \leq 1$ for all $j=1,\dots,N$, and $\lambda_N=1$, from which follows directly that 
	\[
	\norm{K}_2  = \lambda_N = 1, 
	\qquad \mbox{and} \qquad
	\norm{K^2}_2 = \lambda_N^2 = 1.
	\]
	Moreover,
	\begin{equation}\label{eq:OrderSumK}
		\frac{1}{Nh} \sum_{s=1}^{N} \abs{\mathcal{K}_{h,q} (x,x_s)} = \mathcal{O}(1),
	\end{equation}
	and 
	\begin{equation}\label{eq:OrderSumK2}
		\frac{1}{Nh} \sum_{s=1}^{N} \mathcal{K}_{h,q}(x,x_s)^2 = \mathcal{O}(1),
	\end{equation}
	uniformly for $x \in [0,1]$ as $h\to 0$ and $hN\to\infty$, as shown in \citet{klockmann2024efficient}.
	
	Some additional properties of the kernel matrix are gathered in the following lemma.
	\begin{lemma}[Properties of the kernel matrix]
		\label{lem:PropertiesK}
		For $j=1,\dots,\floor{N/2}$, the eigenvalues of $K$ are bounded as follows:
		\begin{equation}\label{eq:BoundsLambdaj}
			\frac{1}{1+(2 \pi h j)^{2q}} \leq \lambda_j \leq \frac{2}{2+(2 \pi h j)^{2q}} \leq 1,
		\end{equation}
		and there exists $c_j \in [2^{1-(2q)^{-1}}, 2]$ such that
		\begin{equation}\label{eq:DecayLambdaj}
			\lambda_j = \frac{1}{1+(c_j \pi h j)^{2q}}.
		\end{equation}
		Moreover, 
		\begin{equation}\label{eq:OrderNormFK2}
			\norm{K^2}_F^2 = \mathcal{O} (h^{-1}),
		\end{equation}
		and
		\begin{equation}\label{eq:OrderTrK2}
			\tr (K^2) = \mathcal{O} (h^{-1}),
		\end{equation}
		as $h\to 0$ and $hN\to\infty$.
	\end{lemma}
	
	\begin{proof}
		Set $j \in \{1, \dots, \floor{N/2}\}$. It is possible to show that for $q \geq 1$
		\[
		\frac{1}{2} \leq \frac{\mbox{sinc} (\pi j / N)^{2q}}{Q_{2q-2}(j/N)} \leq 1,
		\]
		and therefore, by \eqref{eq:DefKappaj}, $ (2\pi j)^{2q}/2 \leq \kappa_j \leq (2 \pi j)^{2q} $, and there exists $c_j \in [2^{1-(2q)^{-1}}, 2]$ such that $\kappa_j = (c_j \pi j)^{2q}$.
		As a consequence, recalling the eigenvalue expression \eqref{eq:DefLambdaj}, it follows
		\[
		\frac{1}{1+(2 \pi h j)^{2q}} \leq \lambda_j \leq \frac{2}{2+(2 \pi h j)^{2q}} \leq 1, 
		\]
		and
		\[
		\lambda_j = \frac{1}{1+(c_j \pi h j)^{2q}},
		\]
		proving the first part of the lemma.
		
		We now look at the Frobenius norm of $K^2$. Thanks to the symmetry properties of the kernel's eigenvalues, we can write
		\[
		\norm{K^2}_F^2 = 2  \sum_{t=1}^{\floor{N/2}} \lambda_t^4 +1,
		\]
		in case of $N$ odd, or 
		\[
		\norm{K^2}_F^2 = 2  \sum_{t=1}^{N/2-1} \lambda_t^4 +1+ \lambda_{N/2}^4,
		\]
		in case of $N$ even.
		Without loss of generality, consider $N$ even. The proof for the case of $N$ odd follows similar steps.
		Bounding the eigenvalues as in \eqref{eq:BoundsLambdaj} and using the Euler-Maclaurin formula, we can write
		\[
		\begin{split}
			\norm{K^2}_F^2 
			& \leq 2 \sum_{t=1}^{N/2} \frac{1}{\{1+(2 \pi h t)^{2q}/2\}^4} + 1 + \frac{1}{\{1 + (\pi h N)^{2q}/2\}^4} \\
			& = 2 \int_{0}^{N/2} \frac{1}{\{1+(2 \pi h t/2^{1/(2q)})^{2q}\}^4} dt + \frac{2}{\{1 + (\pi h N)^{2q}/2\}^4} + 2 R_1 \\
			& = \frac{2^{1/(2q)}}{\pi h} \int_{0}^{2^{-1/(2q)} \pi h N} \frac{1}{(1+x^{2q})^4} dx + \frac{2}{\{1 + (\pi h N)^{2q}/2\}^4}+ 2 R_1 \\
			& \leq \frac{2^{1/(2q)}}{\pi h} \int_{0}^{\infty} \frac{1}{(1+x^{2q})^4} dx + \frac{2}{\{1 + (\pi h N)^{2q}/2\}^4}+ 2 R_1 \\
			& = \frac{1}{\pi h} \frac{2^{1/(2q)} \Gamma \{4-(2q)^{-1}\} \Gamma \{1+(2q)^{-1}\} }{\Gamma(4)} + \frac{2}{\{1 + (\pi h N)^{2q}/2\}^4} + 2R_1,
		\end{split}
		\]
		where $\Gamma$ is the gamma function, for a remainder $R_1$ such that
		\[
		\begin{split}
			\abs{R_1} 
			& = \abs{\int_{0}^{N/2} -\frac{8 \pi h q (2 \pi h t)^{2q-1}}{\{1 + (2 \pi h t)^{2q}/2\}^5} \bigg( x - \lfloor x \rfloor -\frac{1}{2} \bigg) dt} \\
			& \leq 4 \pi h q \int_{0}^{N/2} \frac{(2 \pi h t)^{2q-1}}{\{1 + (2 \pi h t/2^{1/(2q)})^{2q} \}^5} dt
			\leq 4 q \int_{0}^{\infty} \frac{x^{2q-1}}{\{1 + x^{2q}\}^5} dx
			\leq \frac{1}{2}.
		\end{split}
		\]
		Hence, if $h \to 0$ and $h N \to \infty$,
		\[
		\norm{K^2}_F^2 = \mathcal{O} (h^{-1}).
		\]
		
		Finally, we consider the trace of $K^2$. Using \eqref{eq:OrderSumK2}, we get
		\[
		\tr (K^2) = \sum_{t=1}^{N} \frac{1}{(N h)^2} \sum_{s=1}^{N} \mathcal{K}_{h,q} (x_t, x_s)^2 = \sum_{t=1}^{N} \frac{1}{N h} \mathcal{O}(1) = \mathcal{O}(h^{-1}),
		\]
		as $h \to 0$ and $hN\to \infty$, proving \eqref{eq:OrderTrK2}.
	\end{proof}	
	
	\subsection{Proof of Lemma \ref{lem:DataDecomposition}}
	\label{apx:ProofLemmaDataDecomposition}
	
	\begin{proposition}\label{prp:WDecompostion}
		Let $\mathbf{X}\sim\mathcal{N}_n(0_n,\Sigma)$, with spectral density $f\in\mathcal{F}_{\eta}$, $\eta > 0$.
		Set $D$ the $n \times n$ discrete cosine transform matrix.
		Then $W_j= (D_j^\T \mathbf{X})^2$ can be decomposed into two terms as follows
		\[
		W_j = \widetilde{W}_j + \overline{W}_j,
		\]
		where there exists $\mathbf{Z} \sim \mathcal{N}_n(0_n,\mathbf{I}_n)$ such that $\widetilde{W}_j=f(\pi x_j)Z_j^2$ are independent over $j=1,\dots,n$, and
		\[
		\overline{W}_{j} = \big\{ (R^{1/2})_j^\T Z \big\}^2 + 2 \big\{ f^{1/2}(\pi x_j) Z_j \big\} \big\{ (R^{1/2})_j^\T \mathbf{Z} \big\},
		\]
		with 
		\begin{equation}\label{eq:OrderRij}
			(R^{1/2})_{i,j}= \mathcal{O}(n^{-1}),
		\end{equation}
		as $n\to\infty$.
	\end{proposition}
	
	\begin{proof}
		Since $\mathbf{X} \sim \mathcal{N}_n (0_n, \Sigma)$, it follows that $D^\T \mathbf{X} \sim \mathcal{N}_n(0_n,F)$ where $F=D^\T \Sigma D$ satisfies
		\[
		F_{i,j} = f(\pi x_j) \delta_{i,j} + \frac{1 + (-1)^{\abs{i-j}}}{2} \mathcal{O} \bigg( \frac{1}{n} + \frac{\log n}{n^{\eta}} \bigg),
		\]
		where $\delta_{i,j}$ is the Kronecker delta \citep[see][]{klockmann2024efficient}.
		Moreover, there exists $\mathbf{Z} \sim \mathcal{N}_{n}(0_n,\mathbf{I}_{n})$ standard multivatiate Gaussian, such that $D^\T \mathbf{X} = F^{1/2} \mathbf{Z}$.
		Define $\widetilde{F} \in \mathbb{R}^{n \times n}$ diagonal matrix such that $\widetilde{F}_{i,j} = f(\pi x_j) \delta_{i,j}$, and set $E = F - \widetilde{F}$, so that $E_{i,j} = \mathcal{O} (n^{-1} + n^{-\eta} \log n)$.
		
		For $j = 1, \dots, n$ we can write
		\[
		W_{j} = (D_j^\T \mathbf{X})^2 = \big\{ (F^{1/2})_j^\T \mathbf{Z} \big\}^2 = \widetilde{W}_{j} + \overline{W}_{j},
		\]
		where we have set
		\[
		\widetilde{W}_{j} = \big\{(\widetilde{F}^{1/2})_j^\T \mathbf{Z}\big\}^2 = f(\pi x_j) Z_j^2,
		\]
		and
		\[
		\overline{W}_{j} = \big\{ (R^{1/2})_j^\T \mathbf{Z} \big\}^2 + 2 \big\{(\widetilde{F}^{1/2})_j^\T \mathbf{Z}\big\} \big\{ (R^{1/2})_j^\T \mathbf{Z} \big\},
		\]
		where $R^{1/2}= F^{1/2} - \widetilde{F}^{1/2}$.
		Observe that $\widetilde{W}_{j}$ are independent of each other.
		
		Since $F=\widetilde{F}+E$, we can study the order of $R^{1/2}$ by expanding $F^{1/2}$ around $\widetilde{F}^{1/2}$ using the perturbation theory of Hermitian matrices \citep[see][]{sakurai2020modern}.
		We obtain
		\[
		F^{1/2} = \widetilde{F}^{1/2} + \frac{1}{2} \mbox{diag} ( e_i^\T E e_i ) + \sum_{i=1}^{n} \sum_{j\neq i} \frac{\big\{ \lambda_i^{(F)} \big\}^{1/2}}{\lambda_i^{(F)} - \lambda_j^{(F)}} e_i^\T E e_j e_i e_j^\T + \sum_{i=1}^{n} \sum_{j\neq i} \frac{\big\{ \lambda_i^{(F)} \big\}^{1/2}}{\lambda_i^{(F)} - \lambda_j^{(F)}} e_j^\T E e_i e_j e_i^\T + G,
		\]
		where $G_{i,j}=\mathcal{O}(n^{-2})$. 
		By definition, $R^{1/2}=F^{1/2}- \widetilde{F}^{1/2}$, from where we get
		\[
		(R^{1/2})_{j,i} = \mathcal{O}(n^{-1}).
		\]
	\end{proof}
	
	\begin{corollary}\label{cor:DataDecomposition}
		Set $\mathbf{X} \sim \mathcal{N}_n (0_n, \Sigma)$ with spectral density $f \in \mathcal{F}_{\eta}$, $\eta > 0$, and $h>0$ such that $h\to 0$ and $hT\to \infty$ with $T=\floor{n^{\nu}}$ for $\nu \in (1-\min\{1,\eta\}/3,1)$. 
		Set $\mathbf{Y}^*$ to be the vector of transformed data as in the proposed procedure. 
		Then the new observations $\mathbf{Y}^*$ can be decomposed as
		\[
		\mathbf{Y}^* = \widetilde{\mathbf{Y}}^* + \overline{\mathbf{Y}},
		\]
		where 
		\[
		\widetilde{Y}_{t}^* = \frac{1}{2^{1/2}} \log \frac{\widetilde{Q}_{t}}{m} - \frac{1}{2^{1/2}} \bigg\{ \phi \Big(\frac{m}{2} \Big) - \log \frac{m}{2} \bigg\},
		\]
		for $m=n/T$, are independent over $t=1,\dots,\widetilde{T}$, and
		\[
		\overline{Y}_t = \frac{1}{2^{1/2}} \log \bigg( 1 + \frac{ \overline{Q}_{t} }{\widetilde{Q}_{t}} \bigg),
		\]
		having set
		\[
		\widetilde{Q}_{t} = \sum_{j=(t-1)m+1}^{tm} \widetilde{W}_{j}
		\qquad \mbox{and} \qquad 
		\overline{Q}_{t} = \sum_{j=(t-1)m+1}^{tm} \overline{W}_{j},
		\]
		for $\widetilde{W}_{j}$ and $\overline{W}_{j}$ as in Proposition \ref{prp:WDecompostion}.
	\end{corollary}
	
	\begin{proof}[of Lemma \ref{lem:DataDecomposition}]
		From Corollary \ref{cor:DataDecomposition}, we can decompose $\mathbf{Y}^*$ as $\mathbf{Y}^* = \widetilde{\mathbf{Y}}^* + \overline{\mathbf{Y}}$.
		
		Since $\widetilde{Y}_{t}^*$ are independent over $t=1,\dots,\widetilde{T}$, by \citet{cai2010nonparametric} we have that 
		\[
		\widetilde{Y}_{t}^* = H \big\{ f(x_t^*)\big\} + \zeta_{t} + \xi_{t} - \frac{1}{2^{1/2}} \bigg\{ \phi \Big(\frac{m}{2} \Big) - \log \frac{m}{2} \bigg\}, \quad t=1,\dots,\widetilde{T},
		\]
		where $(t-1)/\widetilde{T} \leq x_t^* \leq t/\widetilde{T}$, $\zeta_{t} \stackrel{\mathclap{iid}}{\sim} \mathcal{N}(0,m^{-1})$ standard Gaussian random variables, and $\xi_t$ centered independent random variables. As in Lemma 4 by \cite{klockmann2024efficient}, it is possible to show that $\xi_t$ has decaying moments
		\[
		\E (\abs{\xi_t}^l) \leq c''_l (\log m)^{2l} \{ m^{-l} + (T^{-1} + T^{-1} n^{1-\eta} \log n)^l \}, \quad l > 1,
		\]
		for a constant $c''_l=c''_l(l, \eta, M_1)$ independent of $n$, and 
		\[
		\widetilde{Y}_{t}^* = \frac{1}{2^{1/2}} \log f(x_t) + \epsilon_{t} + \zeta_{t} + \xi_{t}, \quad t=1,\dots,\widetilde{T},
		\]
		for $\labs{\epsilon_{t}} = \labs{H\{f(x_t^*)\} - H[f\{(t-1)/\widetilde{T}\}]} \leq c' (n^{-1} + n^{-\eta} \log n)$ small deterministic error.
		
		Finally, we obtain the order of $\lnorm{\overline{\mathbf{Y}}}_{\ell_2}$.
		To understand the order of one component $\overline{Y}_{t}$, $t=1,\dots, \widetilde{T}$, we need to establish the order of $\overline{Q}_{t} / \widetilde{Q}_{t}$ first. 
		We have
		\[
		\begin{split}
			\frac{ \overline{Q}_{t} }{\widetilde{Q}_{t}}
			& = \sum_{j=(t-1)m+1}^{tm} \Big[ \big\{ (R^{1/2})_j^\T \mathbf{Z} \big\}^2 + 2 \big\{f^{1/2}(\pi x_j) Z_j \big\} \big\{ (R^{1/2})_j^\T \mathbf{Z} \big\} \Big]  
			\Bigg\{ \sum_{j=(t-1)m+1}^{tm} f(\pi x_j) Z_j^2 \Bigg\}^{-1},
		\end{split}
		\]
		and
		\[
		\begin{split}
			\abs{ \frac{ \overline{Q}_{t} }{\widetilde{Q}_{t}} }
			& \leq  \frac{ \sum_{j=(t-1)m+1}^{tm}  \big\{ (R^{1/2})_j^\T \mathbf{Z} \big\}^2  }{\sum_{j=(t-1)m+1}^{tm}  f(\pi x_j) Z_j^2 } + \Bigg[ \frac{ \sum_{j=(t-1)m+1}^{tm} \big\{ (R^{1/2})_j^\T \mathbf{Z} \big\}^2 }{\sum_{j=(t-1)m+1}^{tm}  f(\pi x_j) Z_j^2 } \Bigg]^{1/2},
		\end{split}
		\]
		therefore we consider the order of $\sum_{j} \{ (R^{1/2})_j^\T \mathbf{Z} \}^2$ and $\sum_{j}  f(\pi x_j) Z_j^2$ separately.
		
		\textit{Order of the numerator:}
		Because of \eqref{eq:OrderRij}, 
		\[
		\Var \Bigg\{ \sum_{i=1}^{n} (R^{1/2})_{i,j} Z_{i}  \Bigg\} = \sum_{i=1}^{n} (R^{1/2})_{i,j}^2 = \mathcal{O}( n^{-1} ).
		\]
		It follows that
		\[
		\sum_{j=(t-1)m+1}^{tm} \big\{ (R^{1/2})_j^\T \mathbf{Z} \big\}^2 
		= \sum_{j=(t-1)m+1}^{tm} \Bigg\{ \sum_{i=1}^{n} (R^{1/2})_{i,j} Z_{i} \Bigg\}^2 = \sum_{j=(t-1)m+1}^{tm} \big\{ \mathcal{O}_p(n^{-1/2}) \big\}^2 = \mathcal{O}_p (T^{-1}).
		\]
		
		\textit{Order of the denominator:}
		Recalling that $f\in\mathcal{F}_{\eta}$,
		\[
		\sum_{j=(t-1)m+1}^{tm} f(\pi x_j) Z_{j}^2 \geq c \sum_{j=(t-1)m+1}^{tm} Z_{j}^2 \sim c \chi^2(m).
		\]
		Now, for $A_n \sim \chi^2(n)$ it follows that $A_n=n+\mathcal{O}_p(\sqrt{n})$, and consequently
		\[
		A_n^{-1} = n^{-1} \bigg\{ 1 + \mathcal{O}_p \bigg(\frac{1}{n^{1/2}} \bigg) \bigg\}^{-1} = n^{-1} \bigg\{ 1 + \mathcal{O}_p \bigg(\frac{1}{n^{1/2}} \bigg) \bigg\} = \mathcal{O}_p \big( n^{-1} \big) ,
		\]
		by Taylor expansion. It follows that, for $n \to \infty$
		\[
		\Bigg\{ \sum_{j=(t-1)m+1}^{tm}  f(\pi x_j) Z_j^2 \Bigg\}^{-1} = \mathcal{O}_p(m^{-1}).
		\]
		
		\textit{Order of $\lnorm{\overline{\mathbf{Y}}}_{\ell_2}$:}
		We obtained 
		\[
		\abs{ \frac{ \overline{Q}_{t} }{\widetilde{Q}_{t}} } = \mathcal{O}_p (n^{-1/2}),
		\]
		for which it follows that as $n \to \infty$	
		\[
		\overline{Y}_{t} = \frac{1}{2^{1/2}} \log \bigg( 1 + \frac{ \overline{Q}_{t} }{\widetilde{Q}_{t}} \bigg) = \frac{1}{2^{1/2}} \frac{ \overline{Q}_{t} }{\widetilde{Q}_{t}} + \mathcal{O}_p \bigg\{ \bigg( \frac{ \overline{Q}_{t} }{\widetilde{Q}_{t}} \bigg)^2 \bigg\}= \mathcal{O}_p (n^{-1/2}),
		\]
		and we can conclude that $\overline{Y}_{t}^2 = \mathcal{O}_p(n^{-1})$, implying $\lnorm{\overline{\mathbf{Y}}}_{\ell_2} = \mathcal{O}_p (m^{-1/2})$.
	\end{proof}
	
	\subsection{Proof of Lemma \ref{lem:Orders}}
	\label{apx:ProofLemmaOrders}
	
	\begin{lemma}\label{lem:Orders}
		Under the assumptions of Theorem \ref{th:MinimaxTest}, the following properties hold, as $h_k \to  0$ and $h_k T \to \infty$.
		\begin{enumerate}[label=(\roman*)]
			\item 
			\[
			\frac{1}{\widetilde{T}} \boldsymbol{\mu}_0^\T \boldsymbol{\mu}_0 = \mathcal{O} \bigg\{ h_1^{2\min\{\eta, 2q\}} + h_2^{2\min\{ \eta, 2q \}} + \Big( \frac{1}{n_1} + \frac{\log n_1}{n_1^{\eta}} \Big)^2 + \Big( \frac{1}{n_2} + \frac{\log n_2}{n_2^{\eta}} \Big)^2 \bigg\}.
			\]
			\item 
			\[
			\frac{\E \big\{ (\boldsymbol{\mu}_0^\T \mathbf{v})^2 \big\}}{\widetilde{T} \norm{\Omega}_2} = \mathcal{O} \bigg\{ h_1^{2\min\{\eta, 2q\}} + h_2^{2\min\{ \eta, 2q \}} + \Big( \frac{1}{n_1} + \frac{\log n_1}{n_1^{\eta}} \Big)^2 + \Big( \frac{1}{n_2} + \frac{\log n_2}{n_2^{\eta}} \Big)^2 \bigg\}.
			\]
			\item 
			\[
			\frac{\norm{\Omega}_2}{\widetilde{T}} = \frac{1}{n_1} + \frac{1}{n_2}.
			\]
			\item 
			\[
			\abs{ \frac{\E ( \boldsymbol{\mu}_j^\T \mathbf{v} \mathbf{v}^\T \mathbf{v} )}{\widetilde{T} \norm{\Omega}_2} }
			\leq \frac{1}{n_1 m_1 h_1} \Bigg(\frac{\norm{\boldsymbol{\mu}_j}_2^2}{\widetilde{T}} \Bigg)^{1/2} + \frac{1}{n_2 m_2 h_2} \Bigg( \frac{\norm{\boldsymbol{\mu}_j}_2^2}{\widetilde{T}} \Bigg)^{1/2},
			\]
			for $j=0,1$, and in particular
			\[
			\frac{\E ( \boldsymbol{\mu}_0^\T \mathbf{v} \mathbf{v}^\T \mathbf{v} )}{\widetilde{T} \norm{\Omega}_2} = \mathcal{O} \bigg( \frac{h_1^{\min\{\eta, 2q\} - 1}}{m_1 n_1} + \frac{h_2^{\min\{\eta, 2q\} - 1}}{m_2 n_2} \bigg) .
			\]
			\item 
			\[
			\frac{ \Var^{1/2} (\mathbf{v}^\T \mathbf{v}) }{\widetilde{T}} = \mathcal{O} \bigg( \frac{1}{n_1 h_1^{1/2}} + \frac{1}{n_2 h_2^{1/2}} \bigg).
			\]
			\item 
			\[
			\frac{\tr (\Omega)}{\widetilde{T}} = \mathcal{O} \bigg( \frac{1}{n_1 h_1} + \frac{1}{n_2 h_2} \bigg).
			\]
		\end{enumerate}
	\end{lemma}
	
	\begin{proof}
		Before analysing each term, observe that under the null hypothesis $g_1 \equiv g_2$, therefore, whenever considering $\Hyp_0$, we indicate both log-spectral densities by $g \in \mathcal{F}_{\eta}$, for $\eta=\eta_1=\eta_2$.
		\begin{enumerate}[label=(\roman*)]
			\item By triangle inequality, we have that
			\[
			\frac{1}{\widetilde{T}} \boldsymbol{\mu}_0^\T \boldsymbol{\mu}_0 = \frac{1}{\widetilde{T}} \norm{K_1 (\mathbf{g} + \boldsymbol{\epsilon}_1) - K_2 (\mathbf{g} +\boldsymbol{\epsilon}_2) }_{\ell_2}^2 \leq \frac{2}{\widetilde{T}} \norm{(K_1 - K_2) \mathbf{g}}^2 + \frac{2}{\widetilde{T}} \norm{K_1 \boldsymbol{\epsilon}_1 - K_2 \boldsymbol{\epsilon}_2 }^2.
			\]
			
			We can bound the second term thanks to the triangle inequality, the properties of the kernel matrices $K_1$ and $K_2$ and the order of $\epsilon_1$ and $\epsilon_2$ as $h_k\to 0$ and $Th_k\to\infty$, obtaining
			\[
			\begin{split}
				\frac{2}{\widetilde{T}}\norm{K_1 \boldsymbol{\epsilon}_1 - K_2 \boldsymbol{\epsilon}_2}^2 
				\leq \frac{4}{\widetilde{T}} \{ \norm{\boldsymbol{\epsilon}_1}_{\ell_2}^2 + \norm{\boldsymbol{\epsilon}_2}_{\ell_2}^2 \} 
				\leq 4 \Bigg\{ c'_1 \bigg( \frac{1}{n_1} + \frac{\log n_1}{n_1^{\eta}} \bigg)^2 + c'_2 \bigg( \frac{1}{n_2} + \frac{\log n_2}{n_2^{\eta}} \bigg)^2 \Bigg\}.
			\end{split}
			\]
			
			For the first term, set $\tilde{\mathbf{g}} = U^\T \mathbf{g}$ the Fourier transform of $\mathbf{g}$, and using symmetry
			\[
			\begin{split}
				\frac{2}{\widetilde{T}} \norm{(K_1 - K_2) \mathbf{g}}^2
				& = \frac{4}{\widetilde{T}} \sum_{t=1}^{\widetilde{T}/2} ( \lambda_{1,t} - \lambda_{2,t} )^2 \tilde{g}_t^2 \\
				& = \frac{4}{\widetilde{T}} \sum_{t=1}^{\widetilde{T}/2} \frac{(h_1^{2q} - h_2^{2q})^2 (c_t \pi t)^{4q}}{\{1+(c_t \pi h_1 t)^{2q}\}^2 \{1+(c_t \pi h_2 t)^{2q}\}^2 } \tilde{g}_t^2,
			\end{split}
			\]
			where we have used \eqref{eq:DecayLambdaj}.
			Now, if $\eta > 2q$
			\[
			\frac{\norm{(K_1 -K_2) \mathbf{g}}^2}{\widetilde{T}} \leq 2 (h_1^{2q} - h_2^{2q})^2 \frac{1}{\widetilde{T}} \sum_{t=1}^{\widetilde{T}/2} (c_t \pi t)^{4q} \tilde{g}_t^2 \leq C_1 (h_1^{2q} - h_2^{2q})^2 \norm{\mathbf{g}^{(2q)}}^2 \leq C_2 \big( h_1^{4q} + h_2^{4q}\big),
			\]
			while if $1/2 < \eta < 2q$
			\[
			\begin{split}
				\frac{\norm{(K_1 -K_2) \mathbf{g}}^2}{\widetilde{T}}
				& \leq 4 \max_{t=1, \dots, \widetilde{T}/2} \bigg[ \frac{(h_1^{2q} - h_2^{2q})^2 (c_t \pi t)^{4q-2\eta}}{\{1+(c_t \pi h_1 t)^{2q}\}^2 \{1+(c_t \pi h_2 t)^{2q}\}^2 } \bigg] \norm{\mathbf{g}^{(\eta)}}^2 \\
				& \leq C_3 h_1^{2\eta} \max_{t=1, \dots, \widetilde{T}/2} \bigg[ \frac{  (c_t \pi h_1 t)^{4q-2\eta}}{\{1+(c_t \pi h_1 t)^{2q}\}^2 } \bigg] + C_3 h_2^{2\eta} \max_{t=1, \dots, \widetilde{T}/2} \bigg[ \frac{ (c_t \pi h_2 t)^{4q-2\eta}}{\{1+(c_t \pi h_2 t)^{2q}\}^2 } \bigg] \\
				& \leq C_4 \big( h_1^{2\eta} + h_2^{2\eta} \big),
			\end{split}
			\]
			obtaining
			\[
			\frac{\norm{(K_1 -K_2) \mathbf{g}}^2}{\widetilde{T}} \leq C_5 \big( h_1^{2\min\{\eta, 2q\}} + h_2^{2\min\{ \eta, 2q \}} \big).
			\]
			
			Finally,
			\begin{equation}\label{eq:OrderE[(mu0^tr)2]}
				\frac{1}{\widetilde{T}} \boldsymbol{\mu}_0^\T \boldsymbol{\mu}_0 = \mathcal{O} \Bigg\{ h_1^{2\min\{\eta, 2q\}} + h_2^{2\min\{ \eta, 2q \}} + \bigg( \frac{1}{n_1} + \frac{\log n_1}{n_1^{\eta}} \bigg)^2 + \bigg( \frac{1}{n_2} + \frac{\log n_2}{n_2^{\eta}} \bigg)^2 \Bigg\}.
			\end{equation}
			
			\item For $j=0,1$
			\[
			\begin{split}
				E \big\{ (\boldsymbol{\mu}_j^\T \mathbf{v})^2 \big\} = E ( \boldsymbol{\mu}_j^\T \mathbf{v} \mathbf{v}^\T \boldsymbol{\mu}_j ) = \boldsymbol{\mu}_j^\T E (\mathbf{v} \mathbf{v}^\T) \boldsymbol{\mu}_j = \boldsymbol{\mu}_j^\T \Omega \boldsymbol{\mu}_j,
			\end{split}
			\]
			and in particular
			\[
			\frac{E \big\{ (\boldsymbol{\mu}_0^\T \mathbf{v})^2 \big\} }{\widetilde{T} \norm{\Omega}_2} = \frac{\boldsymbol{\mu}_0^\T \Omega \boldsymbol{\mu}_0}{\widetilde{T} \norm{\Omega}_2} \leq \frac{1}{\widetilde{T}} \boldsymbol{\mu}_0^\T \boldsymbol{\mu}_0 ,
			\]
			so that
			\[
			\frac{\E \big\{ (\boldsymbol{\mu}_0^\T \mathbf{v})^2 \big\}}{ \widetilde{T} \norm{\Omega}_2} = \mathcal{O} \Bigg\{ h_1^{2\min\{\eta, 2q\}} + h_2^{2\min\{ \eta, 2q \}} + \bigg( \frac{1}{n_1} + \frac{\log n_1}{n_1^{\eta}} \bigg)^2 + \bigg( \frac{1}{n_2} + \frac{\log n_2}{n_2^{\eta}} \bigg)^2 \Bigg\}.
			\]
			
			\item Recall by the definition of $\Omega$, set $\boldsymbol{\omega}_k = \boldsymbol{\zeta}_k + \boldsymbol{\xi}_k$,
			\[
			\Omega = K_1 \Var (\boldsymbol{\omega}_1) K_1 + K_2 \Var (\boldsymbol{\omega}_2) K_2 = \frac{K_1^2}{m_1} + \frac{K_2^2}{m_2},
			\]
			therefore, since the two matrices are simultaneously diagonalizable
			\[
			\norm{\Omega}_2 = \lambda_{\widetilde{T}}^{(\Omega)} = \frac{\lambda_{\widetilde{T}}^{(K_1^2)}}{m_1} + \frac{\lambda_{\widetilde{T}}^{(K_2^2)}}{m_2} = \frac{1}{m_1} + \frac{1}{m_2},
			\]
			from which it follows
			\[
			\frac{\norm{\Omega}_2}{\widetilde{T}} = \frac{1}{n_1} + \frac{1}{n_2}.
			\]
			
			\item For $j=0,1$
			\[
			\begin{split}
				\E ( \boldsymbol{\mu}_j^\T \mathbf{v} \mathbf{v}^\T \mathbf{v} )
				& = \sum_{t,s=1}^{\widetilde{T}} \mu_{j,t} \E ( v_t v_s^2) \\
				& = \sum_{t,s,u,i,l=1}^{\widetilde{T}} \mu_{j,t} (K_1)_{t,u} (K_1)_{s,i} (K_1)_{s,l} \E ( \omega_{1,u} \omega_{1,i} \omega_{1,l}) + \\
				& \quad - \sum_{t,s,u,i,l=1}^{\widetilde{T}} \mu_{j,t} (K_2)_{t,u} (K_2)_{s,i} (K_2)_{s,l} \E ( \omega_{2,u} \omega_{2,i} \omega_{2,l}) .
			\end{split}
			\]
			Focusing on one of the two terms, thanks to independence, 
			\[
			\begin{split}
				& \sum_{t,s,u,i,l=1}^{\widetilde{T}} \mu_{j,t} (K_1)_{t,u} (K_1)_{s,i} (K_1)_{s,l} \E ( \omega_{1,u} \omega_{1,i} \omega_{1,l} ) \\
				& = \sum_{t,s,u=1}^{\widetilde{T}} \mu_{j,t} (K_1)_{t,u} (K_1)_{s,u}^2 \E ( \omega_{1,u}^3 ) \\
				& \leq C_6 \frac{1}{m_1^3} \sum_{t=1}^{\widetilde{T}} \abs{\mu_{j,t}} \sum_{u=1}^{\widetilde{T}} \abs{(K_1)_{t,u}} \bigg\{\sum_{s=1}^{\widetilde{T}} (K_1)_{s,u}^2 \bigg\} \\
				& \leq C_7 \frac{\norm{\boldsymbol{\mu}_j}_1}{m_1^3 \widetilde{T} h_1} \\
				& \leq C_7 \frac{1}{m_1^3 h_1} \Bigg( \frac{\norm{\boldsymbol{\mu}_j}_2^2}{\widetilde{T}} \Bigg)^{1/2},
			\end{split}
			\]
			where we have used \eqref{eq:OrderSumK}, and proceeding similarly with the second term we obtain
			\[
			\begin{split}
				\abs{ \frac{\E ( \boldsymbol{\mu}_j^\T \mathbf{v} \mathbf{v}^\T \mathbf{v} )}{\widetilde{T} \norm{\Omega}_2} }
				& \leq C_8 \frac{1}{m_1^3 h_1} \frac{1}{\widetilde{T} \norm{\Omega}_2 } \Bigg( \frac{\norm{\boldsymbol{\mu}_j}_2^2}{\widetilde{T}} \Bigg)^{1/2} + C_8 \frac{1}{m_2^3 h_2} \frac{1}{\widetilde{T} \norm{\Omega}_2 } \Bigg( \frac{\norm{\boldsymbol{\mu}_j}_2^2}{\widetilde{T}} \Bigg)^{1/2} \\
				& \leq C_9 \frac{1}{n_1 m_1 h_1} \Bigg( \frac{\norm{\boldsymbol{\mu}_j}_2^2}{\widetilde{T}} \Bigg)^{1/2} + C_9 \frac{1}{n_2 m_2 h_2} \Bigg( \frac{\norm{\boldsymbol{\mu}_j}_2^2}{\widetilde{T}} \Bigg)^{1/2} ,\\
			\end{split}
			\]
			where we have used the fact that $\norm{\Omega}_2 = m_1^{-1}+m_2^{-1} \geq m_k^{-1}$ for any $k=1,2$. Recalling \eqref{eq:OrderE[(mu0^tr)2]} we get
			\begin{equation}\label{eq:OrderE[mu0^trr^tr]}
				\frac{\E ( \boldsymbol{\mu}_0^\T \mathbf{v} \mathbf{v}^\T \mathbf{v} )}{\widetilde{T} \norm{\Omega}_2} = \mathcal{O} \bigg( \frac{h_1^{\min\{\eta,2q\} - 1}}{m_1 n_1} + \frac{h_2^{\min\{\eta,2q\} - 1}}{m_2 n_2} \bigg).
			\end{equation}
			
			\item We have,
			\[
			\Var \big( \mathbf{v}^\T \mathbf{v} \big)
			= \Var\big\{ (K_1 \boldsymbol{\omega}_1)^\T (K_1 \boldsymbol{\omega}_1) \big\} + \Var \big\{(K_2 \boldsymbol{\omega}_2)^\T (K_1 \boldsymbol{\omega}_2) \big\} + 4 \Var \big\{(K_1 \boldsymbol{\omega}_1)^\T (K_2 \boldsymbol{\omega}_2)\big\} .
			\]
			
			We first focus on one of the first two terms, $\Var\big\{ (K_1 \boldsymbol{\omega}_1)^\T (K_1 \boldsymbol{\omega}_1) \big\}$, and we get
			\[
			\begin{split}
				\Var\big\{ (K_1 \boldsymbol{\omega}_1)^\T (K_1 \boldsymbol{\omega}_1) \big\}
				& = 2 \sum_{s,t=1}^{\widetilde{T}} \sum_{u=1}^{\widetilde{T}} \sum_{j=1}^{\widetilde{T}} (K_1)_{t,u} (K_1)_{t,j} (K_1)_{s,u} (K_1)_{s,j} \mbox{Cov} (\omega_{1,u} \omega_{1,j}, \omega_{1,u} \omega_{1,j}) + \\
				& \quad - \sum_{s,t=1}^{\widetilde{T}} \sum_{u=1}^{\widetilde{T}} (K_1)_{t,u}^2 (K_1)_{s,u}^2 \mbox{Cov} (\omega_{1,u}^2, \omega_{1,u}^2) \\
				& \leq C_{10} \frac{1}{m_1^2} \sum_{s,t=1}^{\widetilde{T}} \Biggl( \sum_{u=1}^{\widetilde{T}} (K_1)_{t,u} (K_1)_{s,u} \Biggr) \Biggl( \sum_{j=1}^{\widetilde{T}} (K_1)_{t,j} (K_1)_{s,j} \Biggr) \\
				& = C_{10} \frac{1}{m_1^2} \sum_{s,t=1}^{\widetilde{T}} \Biggl( \sum_{u=1}^{\widetilde{T}} (K_1)_{t,u} (K_1)_{s,u} \Biggr)^2 \\
				& = C_{10} \frac{1}{m_1^2} \norm{K_1^2}_F^2 \\
				& = \mathcal{O} \bigg( \frac{1}{m_1^2 h_1} \bigg),
			\end{split}
			\]
			where we have used that
			\[
			\norm{K_1^2}_F^2 = \mathcal{O} \big( h_1^{-1} \big),
			\]
			as shown in Lemma \ref{lem:PropertiesK}. 
			
			Proceeding analogously, it is possible to show that
			\[
			\Var\big\{ (K_2 \boldsymbol{\omega}_2)^\T (K_2 \boldsymbol{\omega}_2) \big\} = \mathcal{O} \bigg( \frac{1}{m_2^2 h_2} \bigg), 
			\qquad \mbox{and} \qquad
			\Var \big\{(K_1 \boldsymbol{\omega}_1)^\T (K_2 \boldsymbol{\omega}_2)\big\} = \mathcal{O} \bigg( \frac{1}{\min_k \{m_k^2 h_k\}} \bigg),
			\]
			where the latter follows from the fact that
			\[
			\tr (K_1^2 K_2^2) = \mathcal{O} \bigg( \frac{1}{h_1 + h_2} \bigg),
			\]
			which can be shown similarly to \eqref{eq:OrderNormFK2}.
			
			Finally,
			\[
			\frac{ \Var^{1/2} ( \mathbf{v}^\T \mathbf{v} ) }{\widetilde{T}} = \mathcal{O} \bigg( \frac{1}{n_1 h_1^{1/2}} + \frac{1}{n_2 h_2^{1/2}} \bigg).
			\]
			
			\item We obtain the order of $\tr (\Omega) / \widetilde{T}$. Since $K_1$ and $K_2$ are simultaneously diagonalizable, then we have that 
			\[
			\tr (\Omega) = \frac{1}{m_1} \tr (K_1^2) + \frac{1}{m_2} \tr (K_2^2).
			\]
			Now, from the properties of the kernel matrices in Appendix \ref{apx:Kernel matrix} it follows
			\[
			\frac{\tr (\Omega)}{\widetilde{T}} = \mathcal{O} \bigg( \frac{1}{n_1 h_1} + \frac{1}{n_2 h_2} \bigg).
			\]
		\end{enumerate}
	\end{proof}
	
	\subsection{Proof of Lemma \ref{lem:Sum2Integral}}
	
	\begin{lemma}\label{lem:Sum2Integral}
		Let $f \in \mathcal{F}_{\eta}$, with $\eta > 1/2$. Then
		\[
		\frac{1}{N} \sum_{t=1}^{N} f \bigg( \frac{t}{N}\bigg) = \int_{0}^{1} f(x) dx + \mathcal{O}(N^{-\eta}),
		\]
		as $N \to \infty$.
	\end{lemma}
	
	\begin{proof}
		We first consider $\eta > 1$, and we show that
		\[
		\sum_{t=1}^{N} f \bigg( \frac{t}{N}\bigg) - N \int_{0}^{1} f(x) dx = \mathcal{O}(N^{-\eta+1}).
		\]
		Writing the Fouruier expansion of $f$, we get
		\[
		\begin{split}
			\sum_{t=1}^{N} f \bigg( \frac{t}{N}\bigg) - N \int_{0}^{1} f(x) dx
			& = \sum_{t=1}^{N} \sum_{j \in \mathbb{Z}} \hat{f}(j) e^{2 \pi i j \frac{t}{N}} - N \int_{0}^{1} f(x) dx \\
			& = N \hat{f}(0) + \sum_{j \in \mathbb{Z}\setminus \{0\}} \hat{f}(j) \sum_{t=1}^{N} e^{2 \pi i j \frac{t}{N}} - N \hat{f}(0) \\
			& = \sum_{j \in \mathbb{Z}\setminus \{0\}} \hat{f}(j) N \mathbf{1}_{N \mathbb{Z}}(j) \\
			& = N \sum_{j \neq 0} \hat{f}(jN).
		\end{split}
		\]
		By the decay of the Fourier coefficients of $f$, we obtain the result
		\[
		\abs{ N \sum_{j \neq 0} \hat{f}(jN) } \leq C  N \sum_{j > 0} (jN)^{-\eta} = C  N^{-\eta +1} \sum_{j > 0} j^{-\eta} = \mathcal{O}(N^{-\eta+1}).
		\]
		
		Consider now $ 1/2 < \eta \leq 1$. We look at the following difference
		\[
		\begin{split}
			\frac{1}{N} \sum_{t=1}^{N} f \bigg( \frac{t}{N}\bigg) - \int_{0}^{1} f(x) dx 
			& = \frac{1}{N} \sum_{t=1}^{N} f \bigg( \frac{t}{N}\bigg) - \frac{1}{N} \int_{0}^{N} f \bigg( \frac{x}{N}\bigg) dx \\
			& = \frac{1}{N} \sum_{t=1}^{N} f \bigg( \frac{t}{N}\bigg) - \frac{1}{N} \sum_{t=1}^{N} \int_{t-1}^{t} f \bigg( \frac{x}{N}\bigg) dx \\
			& = \frac{1}{N} \sum_{t=1}^{N} \int_{t-1}^{t} \bigg\{ f \bigg( \frac{t}{N}\bigg) - f \bigg( \frac{x}{N}\bigg) \bigg\} dx, 
		\end{split}
		\]
		and taking the absolute value, we get
		\[
		\begin{split}
			\abs{\frac{1}{N} \sum_{t=1}^{N} f \bigg( \frac{t}{N}\bigg) - \int_{0}^{1} f(x) dx }
			& \leq \frac{1}{N} \sum_{t=1}^{N} \int_{t-1}^{t} \abs{ f \bigg( \frac{t}{N}\bigg) - f \bigg( \frac{x}{N}\bigg) } dx \\
			& \leq \frac{M_1}{N} \sum_{t=1}^{N} \int_{t-1}^{t} \abs{\frac{t - x}{N} }^{\eta} dx \\
			& \leq \frac{M_1}{N^{1+\eta}} \sum_{t=1}^{N} \int_{t-1}^{t} 1 dx \\
			& = \frac{M_1}{N^{\eta}}.
		\end{split}
		\]
	\end{proof}
	
	\subsection{Further simulations and type I error}
	\label{apx:TypeIerrorSimulations}
	
	For completeness, we report the power of our test in the same settings of Section \ref{sec:SimulationStudy}, when the smoothing parameter is selected with maximum likelihood (ML) instead of generalised cross-validation (GCV). For a measure of comparison, we report the power of the wavelet test by \citet{decowski2015wavelet} as well. 
	
	\begin{figure}[!htb]
		\centering
		\begin{subfigure}{0.3\textwidth}
			\centering
			\includegraphics[width=\linewidth]{sdf1_4x4}
			\caption{Setting 1, spectral densities}
			\label{fig:spl_case1_sdf}
		\end{subfigure}
		\hfill
		\begin{subfigure}{0.3\textwidth}
			\centering
			\includegraphics[width=\linewidth]{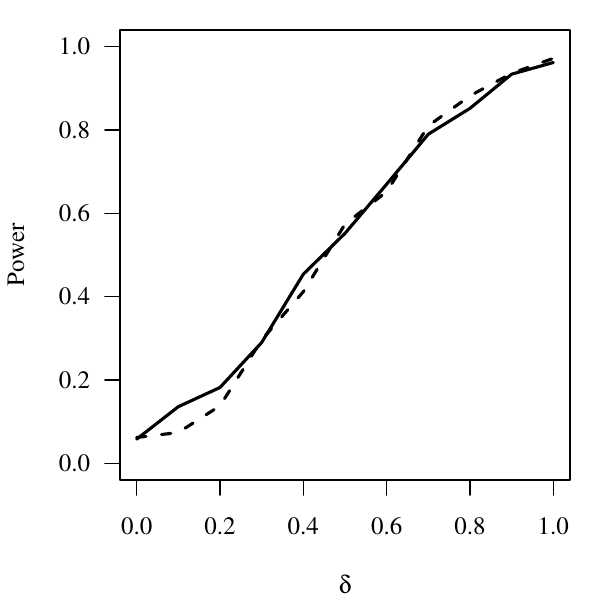}
			\caption{Setting 1, scenario (A)}
			\label{fig:case1_set1_ml}
		\end{subfigure}
		\hfill
		\begin{subfigure}{0.3\textwidth}
			\centering
			\includegraphics[width=\linewidth]{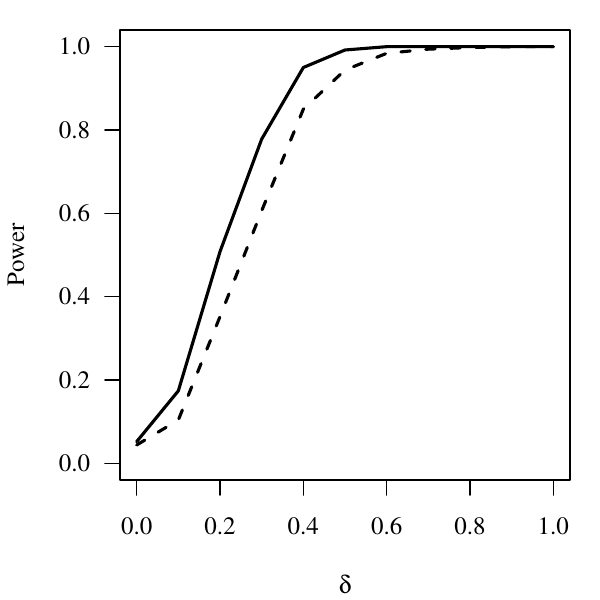}
			\caption{Setting 1, scenario (B)}
			\label{fig:case1_set2_ml}
		\end{subfigure}
		
		\vspace{1em}
		
		\begin{subfigure}{0.3\textwidth}
			\centering
			\includegraphics[width=\linewidth]{sdf2_4x4}
			\caption{Setting 2, spectral densities}
			\label{fig:spl_case2_sdf}
		\end{subfigure}
		\hfill
		\begin{subfigure}{0.3\textwidth}
			\centering
			\includegraphics[width=\linewidth]{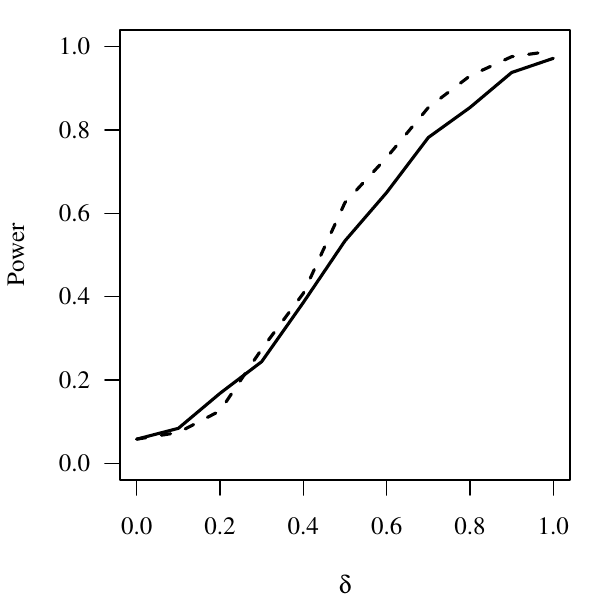}
			\caption{Setting 2, scenario (A)}
			\label{fig:case2_set1_ml}
		\end{subfigure}
		\hfill
		\begin{subfigure}{0.3\textwidth}
			\centering
			\includegraphics[width=\linewidth]{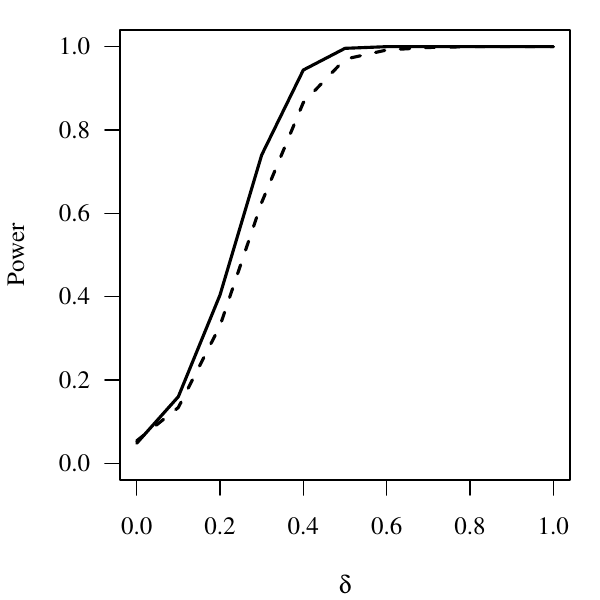}
			\caption{Setting 2, scenario (B)}
			\label{fig:case2_set2_ml}
		\end{subfigure}
		
		\vspace{1em}
		
		\begin{subfigure}{0.3\textwidth}
			\centering
			\includegraphics[width=\linewidth]{sdf3_4x4}
			\caption{Setting 3, spectral densities}
			\label{fig:spl_case3_sdf}
		\end{subfigure}
		\hfill
		\begin{subfigure}{0.3\textwidth}
			\centering
			\includegraphics[width=\linewidth]{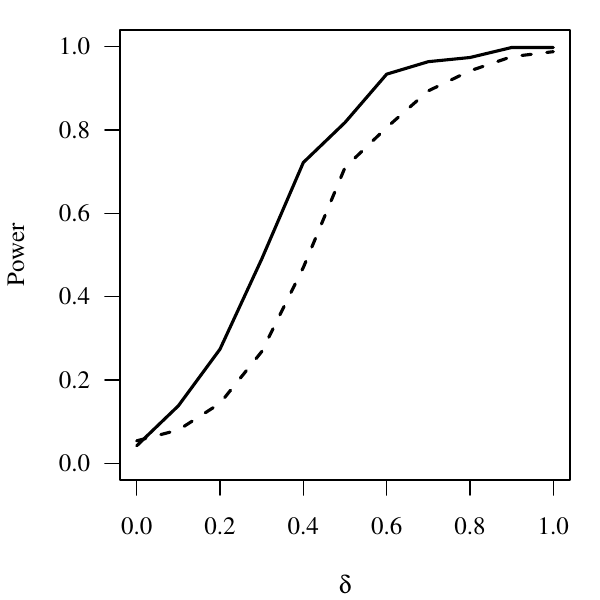}
			\caption{Setting 3, scenario (A)}
			\label{fig:case3_set1_ml}
		\end{subfigure}
		\hfill
		\begin{subfigure}{0.3\textwidth}
			\centering
			\includegraphics[width=\linewidth]{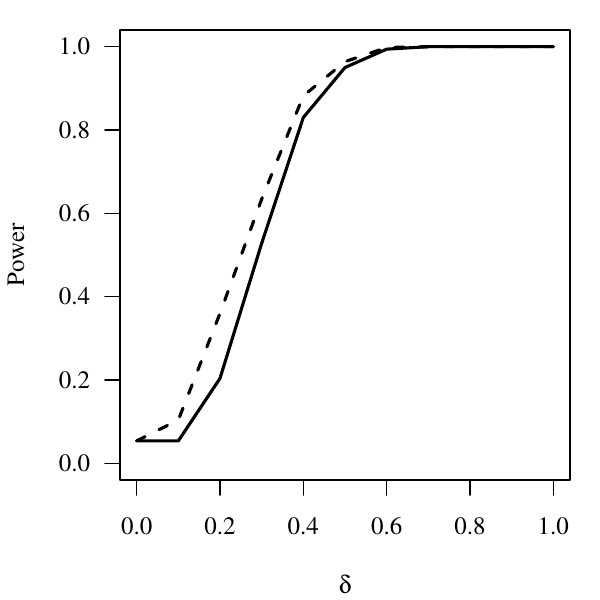}
			\caption{Setting 3, scenario (B)}
			\label{fig:case3_set2_ml}
		\end{subfigure}
		
		\caption{In the first column, the spectral densities $f_1$ (black) and $\widetilde{f}$ (grey). In the second column, the powers of $\psi_{n_1,n_2}$ (solid), and the wavelet-based test (dashed) in scenario (A). In the third column, the powers of $\psi_{n_1,n_2}$ (solid) with maximum likelihood method of smoothing parameter selection, and the wavelet-based test (dashed) in scenario (B).}
		\label{fig:powers_ml}
	\end{figure}
	
	Finally, we report in Table \ref{tab:typeIerrors} the type I errors in the simulations shown in Section \ref{sec:SimulationStudy}. 
	
	\begin{table}[h!]
		\centering
		\caption{Table of type I errors in the simulations.}
		\label{tab:typeIerrors}
		\begin{tabular}{lcccccc}
			& \multicolumn{2}{c}{Setting 1}
			& \multicolumn{2}{c}{Setting 2}
			& \multicolumn{2}{c}{Setting 3} \\
			\cline{2-3} \cline{4-5} \cline{6-7}
			& A & B & A & B & A & B \\
			\hline
			$\psi_{n_1,n_2}$ with GCV & 0.066 & 0.050 & 0.068 & 0.038 & 0.044 & 0.060 \\
			$\psi_{n_1,n_2}$ with ML  & 0.058 & 0.052 & 0.058 & 0.048 & 0.042 & 0.054 \\
			Wavelet test               & 0.062 & 0.044 & 0.058 & 0.054 & 0.054 & 0.054 \\
		\end{tabular}
		
		{\footnotesize}
	\end{table}
	
	\subsection{Comparisons with other tests}
	\label{apx:AdditionalComparisons}
	
	Other available methods for the comparison of two time series are proposed by \citet{dette_paparoditis2009}, \citet{preuss_hildebrandt2013}, and \citet{caiado2012tests}. 
	To compare our test with the tests proposed by them, we perform a ROC analysis, made with $N=500$ samples of time series pairs $X_1$, $X_2$ under the null hypothesis (with same spectral density $f$), and $N=500$ samples of time series pairs $X_1$, $X_2$ under the alternative hypothesis (where $X_1$ has spectral density $f$, while $X_2$ has spectral density $\widetilde{f}$). 
	We consider the following three settings:
	\begin{enumerate}
		\item In the first setting, we assume that $f$ is the spectral density of an AR(1) process with parameter $\phi=0.8$ and standard normal innovations, that is $f(x) = (2 \pi)^{-1} \{1 - 2 \phi \cos(x) + \phi^2\}^{-1}$ for $0 \leq x \leq \pi$. The alternative function $\widetilde{f}$ is the spectral density of a moving averge process MA(1), with moving average parameter $\theta= \phi/(1-\phi^2)^{1/2}$, and standard normal innovations, that is $\widetilde{f} (x) = (2 \pi)^{-1} \{1+2\theta \cos(x) + \theta^2\}$ for $0 \leq x \leq \pi$. 
		\item In the second setting, we assume that $f(x) = 1.44 (2 \pi)^{-1} (\abs{\cos(x/2)}^{5.1} + 0.45)$ for $0 \leq x \leq \pi$, and the alternative function is $\widetilde{f}(x) = 1.44 (2 \pi)^{-1} \{\abs{\cos (x/2 -0.2 \pi ) }^{3.1} + 0.45\}$ for $0 \leq x \leq \pi$. 
		\item In the third setting, we assume that $f(x)=1.44 (2 \pi)^{-1} (\abs{\cos(x/2)}^{5.1}+0.45)$ for $0 \leq x \leq \pi$, and the alternative function is $\widetilde{f}(x) = 0.7 f(x) + 0.3 g(x)$ for $0 \leq x \leq \pi$, where $g(x) = 1.44 (2 \pi)^{-1} \{\abs{\cos (x/2 -0.2 \pi ) }^{3.1} + 0.45\}$. 
	\end{enumerate}	
	The spectral densities of each setting are shown in Figure \ref{fig:spl-sdf}.
	
	\begin{figure}[!htb]
		\centering
		\begin{subfigure}{0.3\textwidth}
			\centering
			\includegraphics[width=\linewidth]{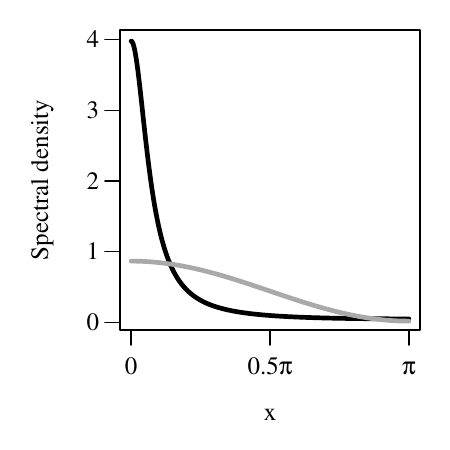}
			\caption{Setting 1, spectral densities}
			\label{fig:spl-case1_sdf}
		\end{subfigure}
		\hfill
		\begin{subfigure}{0.3\textwidth}
			\centering
			\includegraphics[width=\linewidth]{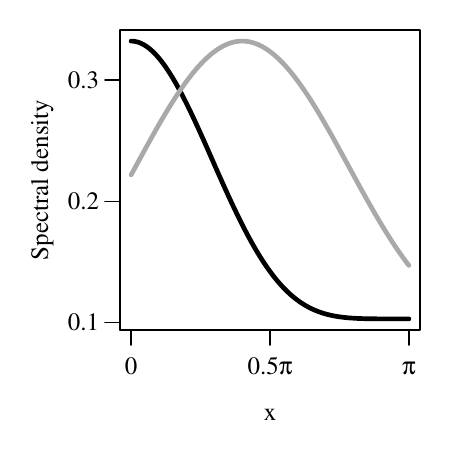}
			\caption{Setting 2, spectral densities}
			\label{fig:spl-case2_sdf}
		\end{subfigure}
		\hfill
		\begin{subfigure}{0.3\textwidth}
			\centering
			\includegraphics[width=\linewidth]{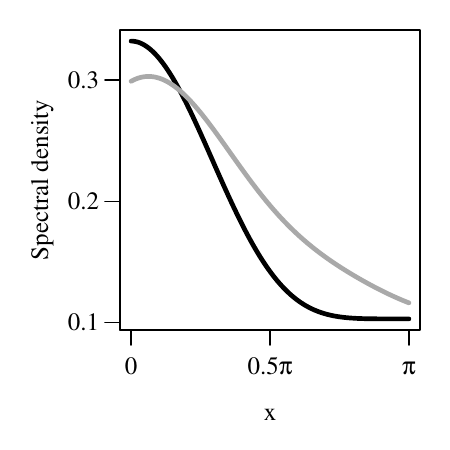}
			\caption{Setting 3, spectral densities}
			\label{fig:spl-case3_sdf}
		\end{subfigure}
		
		\caption{Spectral densities $f$ (black) and $\widetilde{f}$ (grey) for each setting.}
		\label{fig:spl-sdf}
	\end{figure}
	
	Since the test by \citet{dette_paparoditis2009} requires time series with equal length, we update scenario (B) of long time series into a new scenario (B'), in which the time series have equal length ($n_1 = 1024$, and $n_2 = 1024$), and consider it along scenario (A) ($n_1 = 1200$, and $n_2 = 350$).
	For scenario (B'), we set $\nu_1= 0.8$, obtaining the number of bins for our and the wavelet-based test to be $T=256$. 
	Due to the smoothness of the spectral densities, we set $q=4$ for our test. As in the previous comparisons, we select the smoothing parameter for the log-spectral density with generalised cross-validation, and we set $J=1$ for the wavelet-based test.
	All the simulations are performed using \texttt{R} \citep[version 4.2.2, seed 42]{R}. The obtained powers are shown in Figure \ref{fig:powers2}.
	
	\begin{figure}
		\centering
		\begin{subfigure}{0.45\textwidth}
			\centering
			\includegraphics[width=\linewidth]{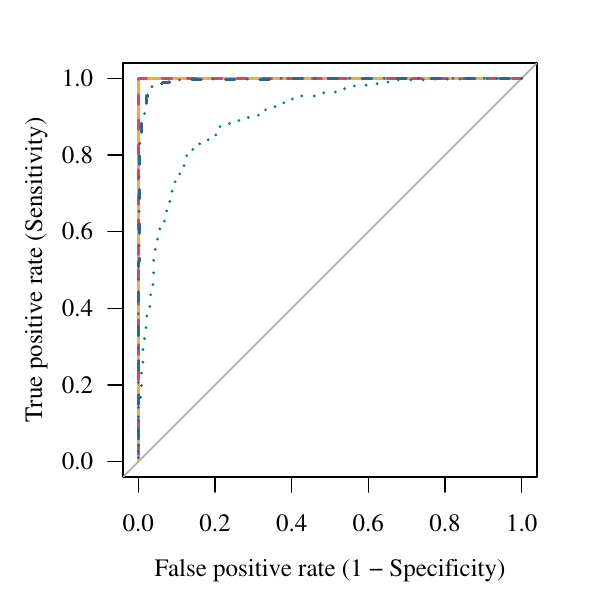}
			\caption{Setting 1, scenario (A)}
			\label{fig:spl_case1_set1}
		\end{subfigure}
		\hfill
		\begin{subfigure}{0.45\textwidth}
			\centering
			\includegraphics[width=\linewidth]{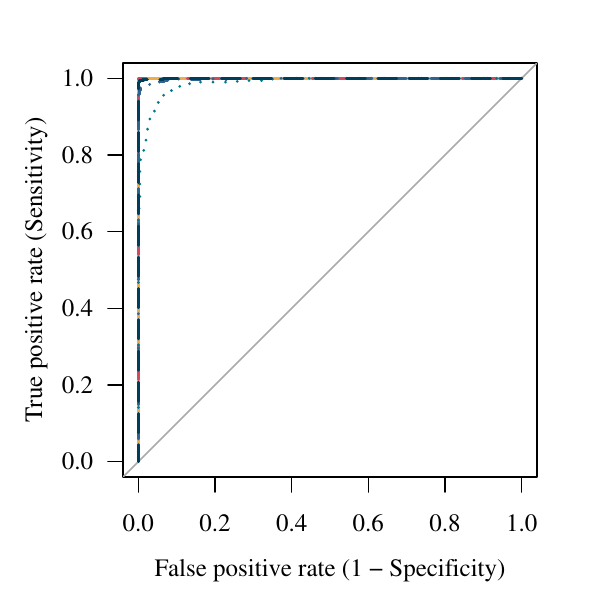}
			\caption{Setting 1, scenario (B')}
			\label{fig:spl_case1_set2}
		\end{subfigure}
		
		\vspace{1em}
		
		\begin{subfigure}{0.45\textwidth}
			\centering
			\includegraphics[width=\linewidth]{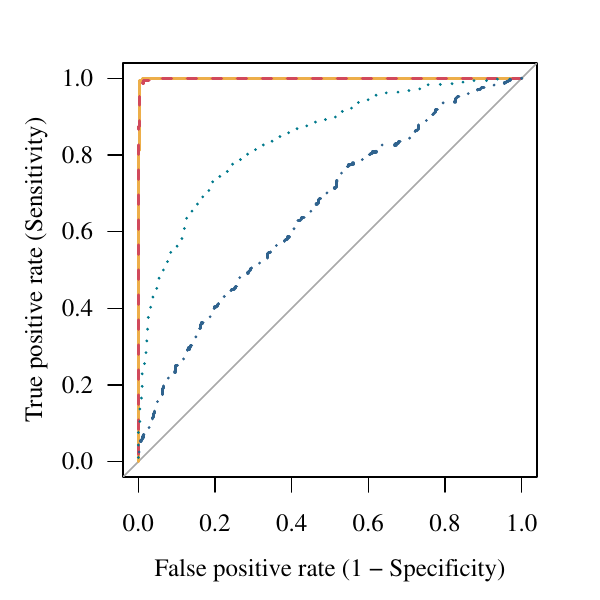}
			\caption{Setting 2, scenario (A)}
			\label{fig:spl_case2_set1}
		\end{subfigure}
		\hfill
		\begin{subfigure}{0.45\textwidth}
			\centering
			\includegraphics[width=\linewidth]{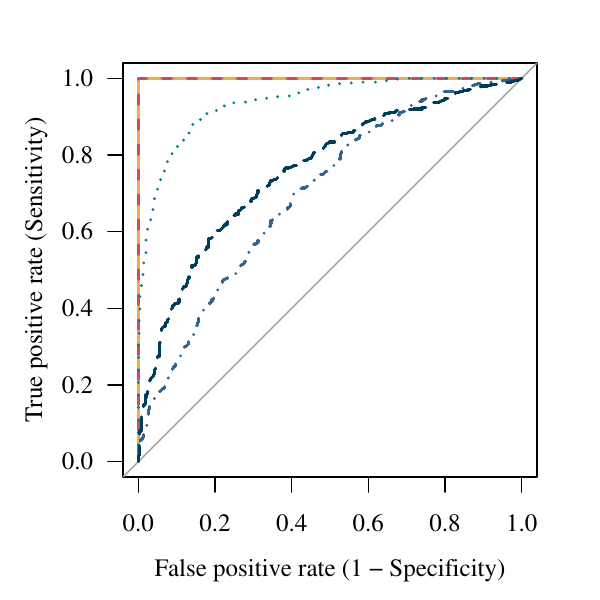}
			\caption{Setting 2, scenario (B')}
			\label{fig:spl_case2_set2}
		\end{subfigure}
		
		\vspace{1em}
		
		\begin{subfigure}{0.45\textwidth}
			\centering
			\includegraphics[width=\linewidth]{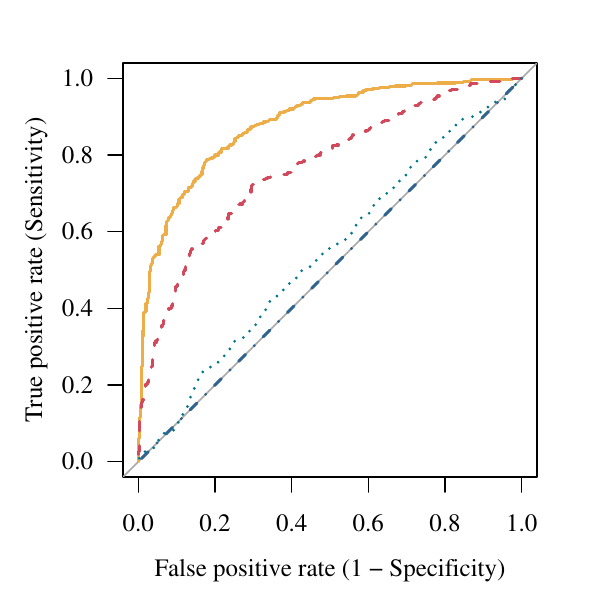}
			\caption{Setting 3, scenario (A)}
			\label{fig:spl_case3_set1}
		\end{subfigure}
		\hfill
		\begin{subfigure}{0.45\textwidth}
			\centering
			\includegraphics[width=\linewidth]{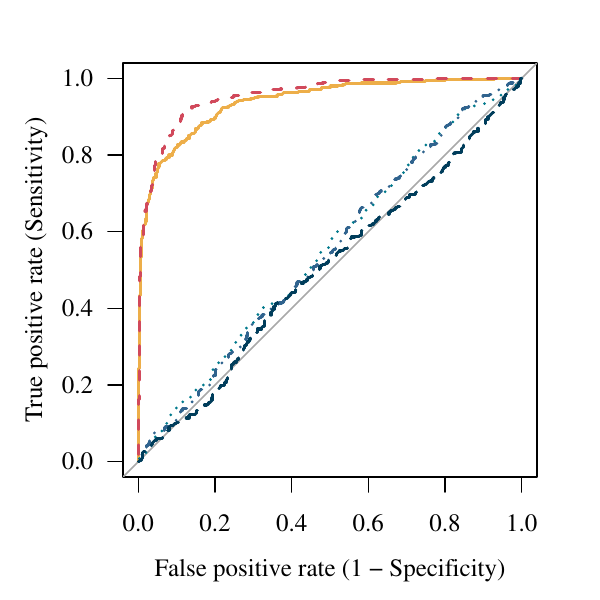}
			\caption{Setting 3, scenario (B')}
			\label{fig:spl_case3_set2}
		\end{subfigure}
		
		\caption{ROC curves of $\psi_{n_1,n_2}$ (solid yellow), the wavelet-based test (dashed red), the test by \citet{preuss_hildebrandt2013} (dotted teal), the test by \citet{caiado2012tests} (dot-dash blue), and the test by \citet{dette_paparoditis2009} (long dash dark blue).}
		\label{fig:powers2}
	\end{figure}
	
	In settings 1 and 2, the ROC curves of both our test and the wavelet-based test overlap, and both achieve perfect classification. Moreover, even when the two spectral densities are difficult to distinguish, as in setting 3, these two tests are already able to address the problem. Our test performs better in setting 3 when the sample sizes are unbalanced, in accordance with the findings in Section \ref{sec:SimulationStudy}.
	
	\bibliographystyle{plainnat}
	\bibliography{Bibliography.bib}
	
\end{document}